\documentclass[10.5pt,a4paper]{article}

\usepackage{graphicx,latexsym,euscript,makeidx,color,bm}
\usepackage{amsmath,amsfonts,amssymb,amsthm,thmtools,mathtools,mathrsfs,enumerate}
\usepackage[colorlinks,linkcolor=blue,anchorcolor=green,citecolor=red]{hyperref}
\usepackage{tikz}
\usetikzlibrary{fit,calc,positioning} %for picture
\tikzstyle{Block}=[rectangle,minimum width=3cm,minimum height=1cm,text centered,text width=5.2cm,draw=black]
\tikzstyle{Implication}=[rectangle,minimum width=3cm,minimum height=1cm,text centered,text width=5.2cm]
\tikzstyle{jian}=[<->, >=stealth]
\usepackage{geometry}
\allowdisplaybreaks[4]
  \def\cA{{\cal A}}  
  \def\cB{{\cal B}}  
  \def\cC{{\cal C}}  
 \def\sD{\mathscr{D}}   
\def\dbE{\mathbb{E}}    
\def\dbF{\mathbb{F}}  \def\cF{{\cal F}}  
    
\def\dbH{\mathbb{H}}  \def\cH{{\cal H}}

  \def\cL{{\cal L}}  
  \def\cM{{\cal M}}

\def\dbP{\mathbb{P}}    
   \def\BQ{{\bm Q}} 
\def\dbR{\mathbb{R}}   \def\BR{{\bm R}} 
\def\dbS{\mathbb{S}}  \def\cS{{\cal S}} \def\BS{{\bm S}} 
\def\dbT{\mathbb{T}}    
 \def\sU{\mathscr{U}}

\def\dbX{\mathbb{X}} \def\sX{\mathscr{X}} \def\cX{{\cal X}}

\def\bx{{\bf x}}     \def\by{{\bf y}}       
\def\ss{\smallskip}             \def\hb{\hbox}
\def\ms{\medskip}              \def\ae{\hb{a.e.}}
        \def\lan{\langle}    \def\as{\hb{a.s.}}
\def\ds{\displaystyle}   \def\ran{\rangle}    
         
\def\no{\noindent}          
         
\def\nn{\nonumber}         
\def\rf{\eqref}            \def\esssup{\hb{esssup}}
\def\cd{\cdot}             \def\cds{\cdots}

\def\deq{\triangleq}     \def\({\Big (}       \def\ba{\begin{aligned}}
\def\les{\leqslant}      \def\){\Big )}       \def\ea{\end{aligned}}
\def\ges{\geqslant}      \def\[{\Big[}        \def\bel{\begin{equation}\label}
          \def\]{\Big]}        \def\ee{\end{equation}}
      \def\q{\quad}        
\def\h{\widehat}         \def\qq{\qquad}

\def\a{\alpha}  \def\G{\Gamma}      \def\Om{\Omega}  \def\om{\omega}
   \def\D{\Delta}   \def\d{\delta}        
\def\z{\zeta}   \def\Th{\Theta}  \def\th{\theta}  \def\Si{\Sigma}  \def\si{\sigma}
\def\f{\varphi} \def\L{\Lambda}  \def\l{\lambda}        \def\e{\varepsilon}
\def\t{\tau}    \def\i{\infty}   \def\k{\kappa}   
\def\ba{\begin{array}}
\def\ea{\end{array}}%------------------------------------------------------------------------------------------------
\def\5n{\negthinspace \negthinspace \negthinspace \negthinspace \negthinspace }
\def\4n{\negthinspace \negthinspace \negthinspace \negthinspace }
\def\3n{\negthinspace \negthinspace \negthinspace }
\def\2n{\negthinspace \negthinspace }
\def\1n{\negthinspace }
\def\esssup{\mathop{\rm esssup}}

\newtheoremstyle{thry}% name
{}      % Space above
{}      % Space below
{\sl}   % Body font
{}      % Indent amount
{\bf}   % Theorem head font
{.}     % Punctuation after theorem head
{.5em}  % Space after theorem head
{}      % Theorem head spec (can be left empty, meaning 'normal')

\theoremstyle{thry}

\newtheorem{theorem}{Theorem}[section]
\newtheorem{proposition}[theorem]{Proposition}
\newtheorem{corollary}[theorem]{Corollary}
\newtheorem{lemma}[theorem]{Lemma}

\theoremstyle{definition}
\newtheorem{definition}[theorem]{Definition}
\newtheorem{example}[theorem]{Example}

\newtheorem{taggedassumptionx}{}
\newenvironment{taggedassumption}[1]
 {\renewcommand\thetaggedassumptionx{#1}\taggedassumptionx}
 {\endtaggedassumptionx}

\theoremstyle{remark}
\newtheorem{remark}[theorem]{Remark}

\def\sqr#1#2{{\vcenter{\vbox{\hrule height.#2pt
              \hbox{\vrule width.#2pt height#1pt \kern#1pt \vrule width.#2pt}
              \hrule height.#2pt}}}}
\makeatletter
   
   \@addtoreset{equation}{section}
   \newcommand{\setword}[2]{%
   \phantomsection
   #1\def\@currentlabel{\unexpanded{#1}}\label{#2}%
   }
\makeatother

\begin{document}
\title{\bf  Linear-Quadratic Optimal Controls for Stochastic Volterra Integral Equations:
 Causal State Feedback and Path-Dependent Riccati Equations}

\author{
Hanxiao Wang\thanks{ College of Mathematics and Statistics, Shenzhen University, Shenzhen 518060, China
 (Email: {\tt hxwang@szu.edu.cn}).}
~~~~
Jiongmin Yong\thanks{Department of Mathematics, University of Central Florida,
                           Orlando, FL 32816, USA (Email: {\tt jiongmin.yong@ucf.edu}).
                          This author is supported in part by NSF Grant DMS-1812921.}
                         ~~~~
Chao Zhou\thanks{ Department of Mathematics, National University of Singapore,
                           Singapore 119076, Singapore (Email: {\tt matzc@nus.} {\tt edu.sg}).
                          This author is supported by Singapore MOE  AcRF Grants R-146-000-271-112
                          and R-146-000-284-114 as well as NSFC Grant 11871364.}
                         }

\maketitle

\no\bf Abstract. \rm A linear-quadratic optimal control problem for a forward stochastic Volterra integral equation (FSVIE, for short) is considered. Under the usual convexity conditions, open-loop optimal control exists, which can be characterized by the optimality system, a coupled system of an FSVIE and a Type-II backward SVIE (BSVIE, for short). To obtain a causal state feedback representation for the open-loop optimal control, a path-dependent Riccati equation for an operator-valued function is introduced, via which the optimality system can be decoupled. In the process of decoupling, a Type-III BSVIE is introduced whose adapted solution can be used to represent the adapted M-solution of the corresponding Type-II BSVIE. Under certain conditions, it is proved that the path-dependent Riccati equation admits a unique solution, which means that the decoupling field for the optimality system is found. Therefore a causal state feedback representation of the open-loop optimal control is constructed. An additional interesting finding is that when the control only appears in the diffusion term, not in the drift term of the state system, the causal state feedback reduces to a Markovian state feedback.

\ms

\no\bf Keywords. \rm Linear-quadratic optimal control,
 stochastic Volterra integral equation, optimality system,
decoupling field, path-dependent Riccati equation, causal state feedback representation.

\ms

\no\bf AMS subject classifications. \rm 93E20, 49N10, 60H20, 45D05.

\section{Introduction}

Let $(\Om,\cF,\dbF,\dbP)$ be a complete filtered probability space, on which a one-dimensional standard Brownian motion $W$ is defined, whose natural filtration augmented by all the $\dbP$-null sets in $\cF$ is denoted by $\dbF\equiv\{\cF_s\}_{s\ges0}$, and let $T>0$ be a fixed time horizon. For any $t\in[0,T]$, let $\dbF^t=\{\cF_s^t\}_{s\ges0}$ with
$$\cF^t_s=\left\{\2n\ba{ll}
\ss\ds\si\(\big\{W(s)-W(t)\bigm|s\ges t\big\}\cup\big\{N\in\cF\bigm|\dbP(N)=0\big\}\),
\qq& s\in[t,T],\\ [2mm]
\ss\ds\si\(\big\{N\in\cF\bigm|\dbP(N)=0\big\}\),\qq&s\in[0,t).\ea\right.$$
Clearly, $\dbF^0=\dbF$. Next, let
$$\sX_t=C([t,T];\dbR^n)=\big\{\bx_t:[t,T]\to\dbR^n~| ~\bx_t(\cd)\hb{ is continous}\big\},\q t\in[0,T].$$
For each $t\in[0,T]$, $\sX_t$ is a Banach space (of some deterministic functions) under the norm:
$$\|\bx(\cd)\|=\sup_{s\in[t,T]}|\bx(s)|.$$
We will use $\bx_t(\cd)$ below to denote an element in $\sX_t$ to emphasize the role played by $t$.
%Clearly,
%
%$$\sX_s\subseteq\sX_t\subseteq\sX_0\equiv C([0,T];\dbR^n),\qq0\les t\les s\les T.$$
%
Next, we introduce
\bel{L}\ds\L=\big\{(t,\bx_t(\cd))~|~t\in[0,T),\,\bx_t(\cd)\in\sX_t\big\}.\ee
For any given {\it free pair} $(t,\bx_t(\cd))\in\L$,
consider the following controlled linear forward stochastic Volterra integral equation (FSVIE, for short) on $[t,T]$:
\begin{align}
 X(s)&=\bx_t(s)+\int_t^s\big[A(s,\t)X(\t)+B(s,\t)u(\t)\big]d\t\nn\\
&\q+ \int_t^s\big[C(s,\t)X(\t)+D(s,\t)u(\t)\big]dW(\t),\q s\in[t,T],\label{state}
\end{align}
where $A,C:\D_*[0,T]\to\dbR^{n\times n}$, $B,D:\D_*[0,T]\to\dbR^{n\times m}$ are deterministic functions satisfying proper conditions, called the {\it coefficients} of the {\it state equation} \rf{state}. Here, $\D_*[0,T]\deq\{(s,r)\bigm|0\les r\les s\les T\}$ is the lower triangle domain. The process $u(\cd)$ is called the {\it control process} which belongs to the space
$$\sU[t,T]=\Big\{u:[t,T]\times\Om\to\dbR^m\bigm|u(\cd)\hb{~is $\dbF^t$-progressively measurable}, ~\dbE\int^T_t|u(\t)|^2d\t<\i\Big\},$$
and the corresponding solution $X(\cd)=X(\cd\,;t,\bx_t(\cd),u(\cd))$ of \rf{state} which uniquely exists under some proper conditions on the coefficients, is called a {\it state process}.
%For convenience, we define
%
%$$X(s)=\bx_t(t),\qq\ s\in[0,t],~\as$$
%
%which is deterministic.
We denote
$$\ba{ll}
\ss\ds\sX[t,T]=\Big\{X:[t,T]\times\Om\to\dbR^n\bigm|X(\cd)\hb{~is $\dbF^t$-progressively measurable on $[t,T]$, continuous},\\
\ss\ds\qq\qq\qq\qq\qq\qq\qq\qq ~\dbE\[\sup_{s\in[t,T]}|X(s)|^2\]<\i\1n\Big\}.\ea$$
To measure the performance of the control $u(\cd)$, we introduce the following cost functional:
\bel{cost}
J(t,\bx_t(\cd);u(\cd))={1\over 2}\dbE_t\Big\{\int_t^T\big[\lan Q(\t)X(\t),X(\t)\ran+\lan R(\t)u(\t),u(\t)\ran\big]d\t+\lan GX(T),X(T)\ran\Big\},\ee
where $G\in\dbS^n$, the set of all $(n\times n)$ symmetric matrices; $Q:[0,T]\to\dbS^n$, and $R:[0,T]\to\dbS^m$ are deterministic functions. The problem that we are going to study
can be stated as follows.

\ms
\no{\bf Problem (LQ-FSVIE).} For any given free pair $(t,\bx_t(\cd))\in\L$, find a control $\bar u(\cd)\in\sU[t,T]$ such that
\bel{inf-J}
J(t,\bx_t(\cd);\bar u(\cd))\les J(t,\bx_t(\cd);u(\cd)),\qq\forall u(\cd)\in\sU[t,T].
\ee

The above is called a  {\it linear-quadratic} (LQ, for short) {\it optimal control problem for FSVIEs.} Any $\bar u(\cd)\in\sU[t,T]$ satisfying \rf{inf-J} is called an {\it (open-loop) optimal control} of Problem (LQ-FSVIE) for the free pair $(t,\bx_t(\cd))$; the corresponding state process $\bar X(\cd)\equiv X(\cd\,;t,\bx_t(\cd),\bar u(\cd))$ is called an {\it optimal state process}; and the function $V(\cd\,,\cd)$, defined by
\bel{value-function}
V(t,\bx_t(\cd))\deq\inf_{u(\cd)\in\sU[t,T]}J(t,\bx_t(\cd);u(\cd)),
\qq\forall(t,\bx_t(\cd))\in\L,\ee
is called the {\it value function} of Problem (LQ-FSVIE). We point out that the value function, $\bx_t(\cd)\mapsto V(t,\bx_t(\cd))$ is defined on the infinite dimensional Banach space $\sX_t$ (of deterministic continuous functions). Moreover, by regarding Problem (LQ-FSVIE) as an optimization of the quadratic functional on $\sU[t,T]$, with the parameter $(t,\bx_t(\cd))$, we expect that the value function should have the following quadratic form:
\bel{def-v-P-In}
V(t,\bx_t(\cd))={1\over2} P(t)\big(\bx_t(\cd),\bx_t(\cd)\big),\qq\forall(t,\bx_t(\cd))\in\L.\ee
Here, for any $t\in[0,T]$, $P(t)$ is a bilinear functional on $\sX_t\times\sX_t$.

\ms

When the coefficients $A(t,s)$, $B(t,s)$, $C(t,s)$, and $D(t,s)$ are independent of $t$ and the free pair $(t,\bx_t(\cd))
\equiv(t,x)$, which is called an {\it initial pair}, for some $x\in\dbR^n$, Problem (LQ-FSVIE) reduces to a classical LQ optimal control for stochastic differential equations (SDEs, for short), denoted by Problem (LQ-SDE). This has been occupying a main part of the center stage for a long history in  control theory. Since we prefer not to make a lengthy survey on the literature of Problem (LQ-SDE), let us just list some books \cite{Anderson1971, Athens1971, Yong-Zhou1999, Bensoussan2018, Sun-Yong20201, Sun-Yong20202}, where good surveys and tutorials along with extensive references (up to that time) can be found. It is well-known by now that (see \cite[Chapter 6]{Yong-Zhou1999} and \cite{Sun-Yong20201}, for example), Problem (LQ-SDE) can be solved by the following three steps in general: (i) By a variational method, the optimality system is derived, which is a coupled forward-backward SDE (FBSDE, for short); (ii) the optimality system is decoupled by introducing the associated Riccati equation, which is solvable under certain conditions; and (iii) the optimal control is represented as a state feedback. This gives a very satisfactory solution to Problem (LQ-SDE).
%It is a reasonable intuition that such steps should also be suitable for Problem (LQ-FSVIE).

\ms

In recent years, FSVIEs have received more and more attention due to its applications in {\it rough volatility} models of mathematical finance, see, for example, Comte--Renault \cite{CR}, Gatheral--Jaisson--Rosenbaum \cite{GJR-2018}, El Euch--Rosenbaum \cite{Euch, ER-2018}, and Viens--Zhang \cite{Viens-Zhang2019}. The optimal control problem for general FSVIEs has been widely studied in the control theory even before the above-mentioned literatures of rough volatility appeared. Under the assumption that the control domain is convex, the maximum principle (MP, for short) for FSVIEs was first established by Yong \cite{Yong2008}, in which the so-called {\it Type-II backward stochastic Volterra integral equations} (BSVIEs, for short) were introduced, as the associated adjoint equations. See \cite{Agram-Oksendal2015,
Shi-Wang-Yong2015,Wang-Zhang2017,Wang2020} for some further results on the MP for FSVIEs. More recently, the dynamical programming principle (DPP, for short) for FSVIEs was first proved by Viens--Zhang \cite{Viens-Zhang2019}, by lifting the state space into the space of continuous functions. We point out that in \cite{Viens-Zhang2019} a functional It\^o's formula for FSVIEs was established, which serves as a fundamental tool in the current paper.

\ms

On the other hand, for linear-quadratic problems of FSVIEs, namely, for Problem (LQ-FSVIE), some early stage of investigation can be found in Chen--Yong \cite{Chen2007}, Yong \cite[Section 5]{Yong2008}, and Wang \cite{Wang2018}, where the authors studied the corresponding MP under different assumptions, but the associated Riccati equation was not concerned at all. In Abi Jaber--Miller--Pham \cite{Abi2021,Abi2021-2}, the associated Riccati equation was derived, which, of course, brought some new insights into the LQ theory of FSVIEs. However, their problem was formulated only in the case of convolution form; that is $\a(s,r)=\a(s-r)$ for $\a(\cd\,,\cd)=A(\cd\,,\cd),B(\cd\,,\cd),
C(\cd\,,\cd),D(\cd\,,\cd)$, and one cannot directly extend the results in \cite{Abi2021,Abi2021-2} to the general LQ problems, because of the limitation of their lift methods. We emphasize that in all the works mentioned above, the general question of how to decouple the optimality system associated with Problem (LQ-FSVIE) has not been touched. In other words, the crucial step (ii) in the standard path of solving Problem (LQ-SDE) mentioned above is completely missing for general case of FSVIEs. As a result, step (iii) for general FSVIEs does not have its foundation. In our opinion,  this step (ii) is more important than solving Problem (LQ-FSVIE) itself, because on one hand, it links the Hamiltonian system and the (path-dependent) HJB equation of controlled FSVIEs; and on the other hand, it provides some important prototype for decoupling general coupled FBSVIEs. We refer the reader to Ma--Protter--Yong \cite{Ma-Protter-Yong1994}, Ma--Yong \cite{Ma-Yong1999}, Yong \cite{Yong2006}, Ma--Wu--Zhang--Zhang \cite{MWZZ-2015}, and Zhang \cite{Zhang-2017}  for the related results in the SDE setting. From this point of view, we may also say that the main objective of the current paper is to explore the decouple of linear FBSVIEs, taking Problem (LQ-FSVIE) as a carrier, which needs a completely new creative method.

\ms

By a variational method and duality principle, we can obtain the following {\it optimality system} associated with Problem (LQ-FSVIE) (see \autoref{theorem:optimality-condition}): Let $(\bar X(\cd),\bar u(\cd))$ be an open-loop optimal pair. Then
\bel{MP}R(s)\bar u(s)+Y^0(s)=0,\q s\in[t,T],\ee
with
\bel{OP0}\left\{
\begin{aligned}
 X(s)&=\f(t,s)+\int_t^s\big[A(s,r)X(r)+B(s,r)\bar u(r)\big]dr\\
&\q+ \int_t^s\big[C(s,r)X(r)+D(s,r)\bar u(r)\big]dW(r),\\
 Y(s)&=Q(s)X(s)+A(T,s)^\top GX(T)+C(T,s)^\top\zeta(s)\\
&\q+\int_s^T\big[A(r,s)^\top Y(r) +C(r,s)^\top Z(r,s)\big]dr-\int_s^T Z(s,r)dW(r),\\
Y^0(s)&=B(T,s)^\top GX(T)+D(T,s)^\top\zeta(s)\\
&\q+\int_s^T\big[B(r,s)^\top Y(r) +D(r,s)^\top Z(r,s)\big]dr-\int_s^T  Z^0(s,r)dW(r),\\
\eta(s)&=GX(T)-\int_s^T\zeta(r)dW(r).
\end{aligned}\right.\ee
Basically, the above is the MP for Problem (LQ-FSVIE). Note that in the above \rf{OP0}, the second equation is a BSDE with unknown $(\eta(\cd),\z(\cd))$; the third is called a Type-II BSVIE with unknown $(Y(\cd),Z(\cd\,,\cd))$, whose main feature is that both $Z(s,r)$ and $Z(r,s)$ appear; the last is a Type-I BSVIE with unknown $(Y^0(\cd),Z^0(\cd\,,\cd))$, whose free term and drift are known processes from the first three equations. It is seen that the system \rf{OP0} is a fully coupled FBSVIE, with the coupling being through \rf{MP}.

\ms

It is not hard to see that the representation of the optimal control obtained from \rf{MP} (assuming $R(s)$ to be invertible for all $s\in[0,T]$) is not practically feasible. The reason is that in determining $Y^0(s)$, future information $\bar X(T)$ of the optimal state process $\bar X(\cd)$ is involved. In the classical LQ theory (either for ODEs or SDEs), under proper conditions, the optimality system (which is a two-point boundary value problem for ODEs, or an FBSDE for SDEs) can be decoupled by the solution to a proper Riccati equation. The state feedback representation of the open-loop optimal control can be obtained as a by-product. See \cite{Yong-Zhou1999} for the standard LQ problems of ODEs and SDEs. The main tool used in the decoupling procedure for SDEs is the (classical) It\^o's formula (and the chain rule for ODEs). However, for FBSVIEs, the classical It\^o's formula is not applicable. By looking at the problem deeper, one realizes that
the decoupling technique is essentially relying on the flow property of the equations, or some kind of semigroup property of the dynamic system. Unfortunately,  the controlled FSVIE  in the optimality system does not satisfy the flow property in the standard sense. In fact, for $t\les r\les s\les T$,
\bel{flowX1}
X(s) \neq X(r) + \int_r^s \big[A(s,\t)X(\t)+B(s,\t)\bar u(\t)\big] d\t + \int_r^s\big[C(s,\t)X(\t)+D(s,\t)\bar u(\t)\big] dW(\t),\ee
due to which, the decoupling method for SDEs (see Ma--Protter-Yong \cite{Ma-Protter-Yong1994}, called the four-step-scheme) cannot be applied here, and the corresponding Riccati type equation is completely unclear from this path. Indeed, this problem has been open for more than ten years (see \cite{Chen2007,Yong2008} for some early suggestions on this topic).

\ms

Recently, Viens--Zhang \cite{Viens-Zhang2019} and Wang--Yong--Zhang \cite{Wang-Yong-Zhang2021} have developed a theory established some relations between BSVIEs and path-dependent PDEs, which are natural and significant extensions of the four-step-scheme (for FBSDEs decoupling \cite{Ma-Protter-Yong1994, Ma-Yong1999}) to the FBSVIEs. These results provide some hopes for our decoupling the optimality system of Problem (LQ-FSVIE), so that getting a practically feasible representation for optimal control becomes possible.

\ms

We now briefly describe the main clue of this paper. Let $(\bar X(\cd),\bar u(\cd))$ be an open-loop optimal pair. First of all, motivated by \cite{Viens-Zhang2019} (see also \cite{Wang-Yong-Zhang2021}), we introduce the following auxiliary process $\bar\cX(\cd\,,\cd)$ with two time variables:
\bel{auxi-pro}\ba{ll}
\ss\ds\bar\cX(s,r)=\bx_t(s)+\2n\int_t^r\3n\big[A(s,\t) \bar X(\t)+B(s,\t)\bar u(\t)\big]d\t+\2n\int_t^r\3n \big[C(s,\t)\bar X(\t)+D(s,\t)\bar u(\t) \big] dW(\t),\\
\ss\ds\qq\qq\qq\qq\qq\qq\qq\qq\qq\qq\qq\qq\qq\qq t\les r\les s\les T.\ea\ee
Then the flow property holds for the state process $\bar X(\cd)$ in the following sense:
\bel{flow-cX}\ba{ll}
\ds\bar X(s)=\bar\cX(s,r)+\int_r^s \big[A(s,\t)\bar X(\t)+B(s,\t)\bar u(\t)\big]d\t+\int_r^s\big[C(s,\t)\bar X(\t)+D(s,\t)\bar u(\t)\big]dW(\t)\\
\ss\ds\qq\qq\qq\qq\qq\qq\qq\qq\qq\qq\qq\qq\qq\qq t\les r\les s\les T.\ea\ee
It is worthy of pointing out that if $r$ is the current time, then for $s\in[r,T]$, $s\mapsto\bar\cX(s,r)$ only depends on the history $\{(\bar X(\t),\bar u(\t))\bigm|\t\in[t,r]\}$ of the optimal pair $(\bar X(\cd),\bar u(\cd))$ and no future information of the state and control is involved. Therefore, we say that $s\mapsto\bar\cX(s,r)$ is {\it causal}. Next, we are trying to express the optimal control in the following manner:
\bel{bar u}\bar u(r)=\Th(r)\bar\cX(\cd\,,r),\qq r\in[t,T],\ee
where $\Th(r):\sX_t\to\dbR^m$ is a linear bounded operator which can be determined by the coefficients of the state equation and the weighting matrix functions in the cost functional. As $s\mapsto\bar\cX(s,r)$ is causal, the above representation implies that the value $\bar u(r)$ of the optimal control $\bar u(\cd)$ at current time $r$ does not involve future information of the corresponding state process $\bar X(\cd)$. Thus, we call the above a {\it casual state feedback representation} of optimal control $\bar u(\cd)$ (see \autoref{thm:Decoupling}). Such a path is basically the analog of Steps (ii) and (iii) for solving classical LQ problems of ODEs and SDEs. The idea is pretty natural. But to achieve the goal, namely to determine the operator $\Th(\cd)$, is by no means trivial.

\ms

We now highlight the main contributions of this paper.

\ms

$\bullet$ Derived the (path-dependent) Riccati equation \rf{Riccati} for the bilinear operator valued function $P(\cd)$, through which, the operator $\Th(\cd)$ can be determined. Note that if we mimic the four-step scheme for FBSDEs (see \cite{Ma-Protter-Yong1994, Ma-Yong1999, Yong-Zhou1999}) trying to decouple the optimality system, we will encounter some difficulties that seem impossible to overcome. To get around this, we first make use of the above flow property, and the functional It\^o's formula established in \cite{Viens-Zhang2019} to derive the HJB equation for the value function $V(t,\bx_t(\cd))$ (see \rf{HJB-PPDE2}) and from that, we correctly identify the Riccati equation for the bilinear operator valued function $P(\cd)$, whose coefficients are path-dependent.

\ms

$\bullet$ Introduced a Type-III BSVIE whose adapted solution can be used to represent the adapted M-solution of the Type-II BSVIEs in the adjoint equation. (see \autoref{Prop:Type-I-II}), which will play a crucial role in decoupling the optimality system via the solution of the Riccati equation \rf{Riccati}. By a Type-III BSVIE, we mean a BSVIE that contains the diagonal value $Z(s,s)$ of $Z(\cd\,,\cd)$ in the drift. Such an equation was introduced by Wang--Yong \cite{Wang-Yong2021} the first time, and has been widely used in \cite{Hernandez2020,HP2,Hamaguchi2020,Lei-Pun2021,Hernandez2021} while studying time-inconsistent optimal control problems. We coin the name of Type-III BSVIEs (the first time) here to distinguish this kind of equations from the other two types of BSVIEs. With such a relation, the optimality condition for Problem (LQ-FSVIE) can be characterized by a Type-III BSVIEs (see \autoref{theorem:optimality-condition-new}). From this, one sees that the decoupling method for the optimality system of Problem (LQ-FSVIE) is significantly different from that for Problem (LQ-SDE).

\ms

$\bullet$ Proved the existence and uniqueness of the {\it strong regular solution} $P(\cd)$ to the path-dependent Riccati equation \rf{Riccati} under the following standard condition:
\bel{standard-condition-In}
Q(s)\ges 0,\q R(s)\ges\l I_m,\q s\in[0,T];\q G\ges 0,\ee
where $\l>0$ is a given constant (see \autoref{def-SRS} and \autoref{theorem:well-posedness-RE}). It follows that the path-dependent HJB equation admits a unique classical solution
and the {\it decoupling field} of the optimality system really exists. Note that for any $t\in[0,T]$, $P(t)$ is a bilinear functional on $\sX_t\times\sX_t$, which is a Banach space depending on $t$ (rather than a Hilbert space).
This feature makes it different from the operator-valued Riccati equation derived from the LQ control problems for
(stochastic) evolution equation (see \cite{Li-Yong1995,Lv2019}). Moreover, we see that the form of \rf{Riccati} is very similar to the classical stochastic Riccati equation, except that the range of $P(\cd)$ is not a Euclidean space. Needless to say, the form of \rf{Riccati} is more natural than the ones derived in \cite{Abi2021,Abi2021-2}.

\ms

$\bullet$ An additional interesting finding is that when the drift term is not controlled, the causal state feedback representation of optimal control will reduce to a state feedback, which means that the value $\bar u(s)$ of the optimal control $\bar u(\cd)$ at the current time $s$ only depends on the state $\bar X(s)$ at the current time. Moreover, using the solution of the path-dependent Riccati equation, we can obtain a  representation for $(\dbE_s[Y(\cd)]|_{[s,T]},Z(\cd,s)|_{[s,T]})$
in the dual space of $\sX_s$ (see \autoref{corollary-YZ}),
by regarding it as a bounded linear functional on $\sX_s$.

\ms

The rest of the paper is organized as follows. Section \ref{sect-preliminary} collects some preliminary results which include the introduction of a Type-III BSVIE and a representation of the adapted M-solution to a Type-II BSVIE. In Section \ref{sect-opti-system}, the optimality system associated with Problem (LQ-FSVIE) is derived.
We introduce the path-dependent Riccati equation in Section \ref{sect-Riccati} and establish the decoupling for the optimality system in Section \ref{sect-decoupling}.
Finally, in Section \ref{sect-wellposedness} the well-posedness of the Riccati equation is established.

%\subsection{Literature review on some related topics}\label{sect-literature}

\section{Preliminaries}\label{sect-preliminary}

\subsection{Basic results for FSVIEs and BSVIEs}

Throughout this paper,  let $\dbR^{n\times m}$ be the Euclidean space consisting of $n\times m$ real matrices,
and $\dbS^n$ be the subset of  $\dbR^{n\times n}$ consisting of symmetric matrices.
%When $m=1$, we simply write $\dbR^{n\times 1}$ as $\dbR^n$.
Let $T>0$ be the time horizon. Denote
$$\D_*[t,T]=\big\{(s,r)\bigm|t\les r\les s\les T\big\},\qq\D^*[t,T]=\big\{(s,r)\bigm|t\les s\les r\les T\big\}.$$
They are the lower and the upper triangle domains in $[t,T]^2$, respectively. For any Euclidean space $\dbH$, we introduce the following spaces: For any $t\in[0,T)$, (with $\cB([t,T])$ being the Borel $\si$-field of $[t,T]$)
$$\ba{ll}
\ss\ds L^\i(t,T;\dbH)=\big\{\f:[t,T]\to\dbH~|~\f~\hb{is essentially bounded}\big\},\\
\ss\ds L_{\cF_T^t}^2(t,T;\dbH)=\Big\{\f:[t,T]\times
\Om\to\dbH\bigm|\f(\cd)~\hb{is $\cB([t,T])\otimes
\cF^t_T$-measruable},~\dbE\int_t^T|\f(\t)|^2d\t<\i\Big\},\\
\ss\ds L_{\dbF^t}^2(t,T;\dbH)
=\Big\{\f(\cd)\in L^2_{\cF_T^t}(t,T;\dbH)\bigm|\f(\cd)~\hb{is $\dbF^t$-progressively measurable on $[t,T]$}\Big\},\\
%
%\ds C_{\dbF^t}([t,T];L^2(\Om;\dbH))
%=\Big\{\f(\cd)\in L_{\dbF^t}^2(t,T;\dbH)\bigm|\f(\cd)~\hb{is continuous from $[t,T]$ to $L^2(\Om;\dbH)$,}\\
%
%\ds\qq\qq\qq\qq\qq\qq\qq\qq\qq\qq\qq\qq\sup_{s\in[t,T]}\big[\dbE|\f(s)|^2\big]<\i \big\},\\
%
\ds L_{\dbF^t}^2(\Om;C([t,T];\dbH))
=\Big\{\f(\cd)\in L_{\dbF^t}^2(t,T;\dbH)\bigm|\f(\cd)~\hb{has continuous paths,}~\dbE\big[
\sup_{s\in[t,T]}|\f(s)|^2\big]<\i \Big\},\\
\ss\ds L_{\dbF^t}^2([t,T]^2;\dbH)=\Big\{\f:[t,T]^2
\times\Om\to\dbH\bigm|\f(s,\cd)\in L_{\dbF^t}^2(t,T;\dbH),\,\ae~s\in[t,T],
~\dbE\int_t^T\3n\int_t^T\3n|\f(s,\t)|^2d\t ds<\i \Big\},\\
\ss\ds L_{\dbF^t}^2(\D^*[t,T];\dbH)
=\Big\{\f:\D^*[t,T]\1n\times\1n\Om\1n\to\1n\dbH\bigm|\f(s,\cd)
\1n\in\1n L_{\dbF^s}^2(s,T;\dbH),\,\ae~s\1n\in\1n[t,T],~ \esssup_{s\in[t,T]}\dbE\2n\int_s^T\3n|\f(s,\t)|^2d\t\1n<\1n\i \Big\}.\ea$$
For the state equation \rf{state} and the cost functional \rf{cost}, we impose the following assumptions.

\begin{taggedassumption}{(H1)}\label{ass:H1}\rm
The coefficients of the state equation $A,C:\D_*[0,T]\to\dbR^{n\times n}$ and $B,D:\D_*[0,T]\to \dbR^{n\times m}$ are (deterministic) bounded and differentiable, with bounded derivatives.
\end{taggedassumption}

\begin{taggedassumption}{(H2)}\label{ass:H2}\rm
The weighting coefficients in the cost functional \rf{cost} satisfy
$$Q\in L^\i(0,T;\dbS^n),\q R\in L^\i(0,T;\dbS^m),\q G\in \dbS^n.$$
\end{taggedassumption}

Applying the results of Ruan \cite{Ruan2020}, we have the following result.

\begin{lemma}\label{lemma:Well-SVIE}
Let {\rm \ref{ass:H1}} hold. Then for any $(t,\bx_t(\cd))\in\L$ and $u(\cd)\in\sU[t,T]$, state equation \rf{state} admits a unique solution $X(\cd)\equiv X(\cd\,;t,\bx_t(\cd),u(\cd))\in L^2_{\dbF^t}(\Om;C([t,T];\dbR^n))$. Moreover, there exists a constant $K>0$, independent of $(t,\bx_t(\cd))\in\L$ and $u(\cd)\in\sU[t,T]$ such that
\begin{align}
\dbE\[\sup_{s\in[t,T]}|X(s)|^2\]\les K\[\sup_{s\in[t,T]}|\bx_t(s)|^2+\dbE\int_t^T|u(s)|^2ds\].
\end{align}
\end{lemma}

%{\bf\color{red}The above needs a proof.}

\ms
From \autoref{lemma:Well-SVIE}, we see that Problem (LQ-FSVIE) is well-formulated under the assumptions \ref{ass:H1} and \ref{ass:H2}. We remark that \autoref{lemma:Well-SVIE} still holds true under some weaker assumptions than \ref{ass:H1}.
Since we mainly focus on the form of the associated Riccati equation (or the decoupling field of the optimality system) in this paper, we prefer not to pursue the weakest possible conditions to simplify our presentation.

\ms

We now recall some fundamental results of the following Type-II linear BSVIEs:
\bel{linear-system-II-Introduction}
Y(s)=\psi(s)+\int_s^T\big[\cA(s,\t)Y(\t)+\cC(s,\t)Z(\t,s)\big]d\t-\int_s^T Z(s,\t)dW(\t),\q s\in[t,T],\ee
which can be found in Yong \cite{Yong2008}.

\begin{definition}\rm
A pair of processes $(Y(\cd),Z(\cd\,,\cd))\in L^2_{\dbF^t}(t,T;\dbR^n)\times L^2_{\dbF^t}([t,T]^2;\dbR^n)$
is called an {\it adapted M-solution} to BSVIE \rf{linear-system-II-Introduction}
if \rf{linear-system-II-Introduction} is satisfied in the usual It\^o sense for the Lebesgue measure almost every $t\les\t\les s\les T$ and, in addition, the following holds:
\bel{M1-solution}
Y(s)=\dbE_r[Y(s)]+\int_r^sZ(s,\t)dW(\t),\qq\ae~t\les r\les s\les T.\ee
\end{definition}

\begin{lemma}\label{lemma:well-ii-bsvie}
Let $\cA(\cd\,,\cd),\cC(\cd\,,\cd)\in L^\i (\D^*[t,T];\dbR^{n\times n})$. Then for any $\psi(\cd)\in L^2_{\cF^t_T}(t,T;\dbR^n)$, BSVIE \rf{linear-system-II-Introduction} admits a unique adapted M-solution $(Y(\cd),Z(\cd\,,\cd))\in L^2_{\dbF^t}(t,T;\dbR^n)\times L^2_{\dbF^t}([t,T]^2;\dbR^n)$.
\end{lemma}

\ms

Let $\cL(\sX_t;\dbR^m)$ be the set of all bounded $\dbR^m$-valued linear functionals on $\sX_t$, with the norm $\|\cd\|_{\cL}$ defined by
$$\|L\|_{\cL}\deq\sup_{\|\bx_t(\cd)\|\les 1}|L\bx_t(\cd)|,\qq\forall L\in\cL(\sX_t;\dbR^m).$$
Let $ L^\i (t,T;\cL(\sX_t;\dbR^m))$ be the set of all $\cL(\sX_t;\dbR^m)$-valued functions defined on $[t,T]$.
In other words, for any $F(\cd)\in L^\i(t,T;\cL(\sX_t;
\dbR^m))$,
\bel{}F(t)\in\cL(\sX_t;\dbR^m),\q\ae~t\in[0,T],\q \hbox{and}\q\esssup_{t\in[0,T]}\|F(t)\|_\cL<\i.
\ee

\begin{definition}\label{def-causal-feedback}\rm Any $\bar\Th(\cd)\in L^\i(t,T;\cL(\sX_t;\dbR^m))$ is called an {\it optimal causal feedback operator} of Problem (LQ-FSVIE) on $[t,T]$ if
\bel{def-causal-feedback1}
J(t,\bx_t(\cd);\bar\Th(\cd)\bar\cX(\cd\,,\cd))\les J(t,\bx_t
(\cd);u(\cd)),\qq\forall u(\cd)\in\sU[t,T],\ee
where $(s,r)\mapsto\bar\cX(s,r)\equiv\bar\cX_r(s)$ is the unique solution to the closed-loop  auxiliary system
\bel{def-causal-feedback2}\ba{ll}
\ss\ds\bar\cX(s,r)=\bx_t(s)+\int_t^r\big[A(s,\t)\bar X(\t)+B(s,\t)\bar\Th(\t)\bar\cX(\cd\,,\t)\big]d\t\\
\ss\ds\qq\qq\qq+\int_t^r\big[C(s,\t)\bar X(\t)+D(s,\t)\bar\Th(\t)\bar\cX(\cd\,,\t)\big]dW(\t),\qq (s,r)\in\D_*[t,T],\ea\ee
and $\bar X(\cd)$ is the unique solution to the closed-loop system
\bel{def-causal-feedback3}\ba{ll}
\ss\ds\bar X(s)=\bx_t(s)+\int_t^s\big[A(s,\t)\bar X(\t)+B(s,\t)\bar\Th(\t)\bar\cX(\cd\,,\t)\big]d\t\\
\ss\ds\qq\qq\qq\q+\int_t^s\big[C(s,\t)\bar X(\t)+D(s,\t)\bar\Th(\t)\bar\cX(\cd\,,\t)\big]dW(\t),\q s\in[t,T].\ea\ee

\end{definition}

\begin{remark}
It is clear that the outcome $\t\mapsto\bar u(\t)=\bar\Th(\t)\bar\cX(\cd\,,\t)$ of the optimal causal state feedback operator $\bar\Th(\cd)$ is an open-loop optimal control. Note that the auxiliary process $\bar\cX_r(\cd)=\bar\cX(\cd\,,r)$ is uniquely determined by the portion $\bar X(\cd)|_{[t,r]}$ of the state process $\bar X(\cd)$. Thus, $\bar u(\cd)$ has a causal state feedback representation. Namely, the value $\bar u(r)$ of the optimal control $\bar u(\cd)$ at any time $r$ does not involve future information of the corresponding state process $\bar X(\cd)$.
\end{remark}

\subsection{Bilinear operators}

For any Banach space $\dbX$, let $\cL^2(\dbX)$ be the set of all  bounded bilinear functionals on  $\dbX\times\dbX$,
with the norm $\|\cd\|_{\cL}$ defined by
$$\|P\|_{\cL^2}\deq\sup_{\|\bx\|,\|\by\|\les 1}|P(\bx,\by)|,\qq\forall P\in\cL^2(\dbX).$$
Suggestively, we usually denote
$$P(\bx,\by)=\lan P\bx,\by\ran=\by^\top P\bx,\qq\bx,\by\in\dbX,$$
where, $P\bx\in\dbX^*$, the dual of $\dbX$, and $\lan\cd\,,\cd\ran$ is the duality pairing between $\dbX^*$ and $\dbX$.
A bilinear functional $P\in\cL^2(\dbX)$ is said to be {\it symmetric} if it satisfies
$$\by^\top P\bx=P(\bx,\by)= P(\by,\bx)=\bx^\top P\by,\qq \forall \bx,\by\in\dbX.$$
Let $\cS(\dbX)$ be the set of all symmetric bilinear functionals on $\dbX$, and $\cS_+(\dbX)$ be the subset of $\cS(\dbX)$ consisting of all non-negative bilinear functionals, that is $P\in\cS_+(\dbX)$, if and only if $P\in\cS(\dbX)$ and
$$P(\bx,\bx)\ges 0,\qq\forall\bx\in\dbX.$$
For any $P\in\cS(\dbX)$, define
$$\|P\|_\cS\deq\sup_{\|\bx\|\les 1}|P(\bx,\bx)|.$$
The following result  shows that $\|\cd\|_{\cS}$ is an equivalent norm of $\|\cd\|_{\cL^2}$ on $\cS(\dbX)$.

\begin{lemma}\label{lem:norm} For any $P\in\cS(\dbX)$, it holds
\bel{lem:norm-main}\|P\|_\cS\les\|P\|_{\cL^2}\les 2\|P\|_\cS.\ee
\end{lemma}

\begin{proof}
On one hand, by the definition, it is clear to see that $\|P\|_\cS\les\|P\|_{\cL^2}$.
On the other hand, since the bilinear operator $P$ satisfies $P(\bx,\by)=P(\by,\bx)$, we have
$$P(\bx+\by,\bx+\by)-P(\bx-\by,\bx-\by)=2P(\bx,\by)+2P(\by,\bx)=4P(\bx,\by),\q \forall\bx,\by\in\dbX.$$
Thus
$$\ba{ll}
\ss\ds|P(\bx,\by)|={|P(\bx+\by,\bx+\by)-P(\bx-\by,\bx-\by)|\over 4}\\
\ss\ds\qq\qq\les\|P\|_\cS {\|\bx+\by\|^2+\|\bx-\by\|^2\over 4}\les\|P\|_\cS [\|\bx\|^2+\|\by\|^2],\qq\forall\bx,\by\in\dbX,\ea$$
which implies
$$\|P\|_{\cL^2}=\sup_{\|\bx\|,\|\by\|\les 1}|P(\bx,\by)|
\les\|P\|_\cS\sup_{\|\bx\|,\|\by\|\les 1}[\|\bx|^2+|\by\|^2] =2 \|P\|_\cS,$$
proving our conclusion. \end{proof}

\ms

\begin{remark}\rm If $\dbX$ is a Hilbert space, by the standard results in functional analysis we have $\|\cd\|_\cS=\|\cd\|_{\cL^2}$ on $\cS(\dbX)$.
When $\dbX$ is merely a Banach space, though  $\|\cd\|_\cS=\|\cd\|_{\cL^2}$ might not hold in general,
we still have the equivalence between $\|\cd\|_\cS$ and $\|\cd\|_{\cL^2}$ on $\cS(\dbX)$.
\end{remark}

We often simply write $\|\cd\|_{\cS}$ (or $\|\cd\|_{\cL^2}$) as $\|\cd\|$ when there is no confusion. If the values $P(\bx,\bx)$ are determined for all $\bx\in\dbX$, then we can extend $L$ to the whole space $\dbX\times\dbX$ by
$$P(\bx,\by)={P(\bx+\by,\bx+\by)-P(\bx-\by,\bx-\by)\over 4},\qq\forall \bx,\by\in\dbX.$$
Clearly, the extension $P\in\cS(\dbX)$; that is $P$ is a symmetric bilinear operator on $\dbX\times\dbX$. By \autoref{lem:norm}, we know that such an extension is unique in the space $\cS(\dbX)$.
This extension will play a crucial role in establishing the well-posedness of the Riccati equation
associated with Problem (LQ-FSVIE).

\subsection{Functional It\^{o}'s formula}

Recall $\L$ from \rf{L}. As in Viens--Zhang \cite{Viens-Zhang2019}, we introduce the following metric:
$${\bf d}\big((t,\bx_t),(t',\bx_{t'}')\big)\deq|t-t'|+\sup_{s\in[0,T]}|\bx_t(s\vee t)-\bx'_{t'}(s\vee t')|,\qq\forall
(t,\bx_t),(t',\bx'_{t'})\in\L.$$
It can be shown that $\L$ is a complete metric space under ${\bf d}$.
For any $t\in[0,T)$, let $\sX_t^+$ be the subspace of $\sX_t$ consists of all $\bx_t(\cd)\in\sX_t$ such that $\bx_t(s)$ is right-differentiable at each $s\in[t,T]$. Further, we denote $\sX^0_t$ to be the subspace of $\sX_t$ consisting of all constant functions. Correspondingly, we denote $\L^+$ and $\L^0$ as follows:
$$\L^+=\{(t,\bx_t(\cd))\in\sX_t\bigm|\bx_t(\cd)\in\sX_t^+\},\qq
\L^0=\{(t,\bx_t(\cd))\in\sX_t\bigm|\bx_t(\cd)\in\sX_t^0\}.$$
Therefore,
$$\sX_t^0\varsubsetneq\sX_t^+\varsubsetneq\sX_t,\qq\L^0\varsubsetneq
\L^+\varsubsetneq\L.$$
Let $C^0(\L)$ denote the set of all functions $v:\L\to\dbR$ which are continuous under ${\bf d}$. For any $v\in C^0(\L)$ and given $(t,\bx_t(\cd))\in\L$, $v(t,\bx_t(\cd))$ takes real values. We denote $v_\bx(t,\bx_t(\cd))$ to be the Fr\'echet derivative of $v(t,\bx_t(\cd))$ with respect to $\bx_t(\cd)$. Namely $v_\bx(t,\bx_t(\cd)):\sX_t\to\dbR$ is the linear functional satisfying the following:
\bel{u_o}
v(t,\bx_t(\cd)+\eta_t(\cd))-v(t,\bx_t(\cd))=v_\bx(t,\bx_t(\cd))(\eta_t(\cd))
+o(\|\eta_t(\cd)\|),\qq\forall\eta_t(\cd)\in\sX_t.\ee
It is clear that the above Fr\'echet derivative can be calculated in the following way, which defines the G\^ateaux derivative:
\bel{gateaux} v_\bx(t,\bx_t(\cd))(\eta_t(\cd))=\lim_{\e\to0}{v(t,\bx_t(\cd)+\e\eta_t(\cd))-v(t,\bx_t(\cd))
\over\e},\qq\forall\eta_t(\cd)\in\sX_t.\ee
Similarly, we define the second order derivative $v_{\bx\bx}(t,\bx_t(\cd))$ as a bilinear functional on $\sX_t\times\sX_t$:
\bel{u_oo}\ba{ll}
\ss\ds v_\bx(t,\bx_t(\cd)+\eta_t(\cd))(\eta'_t(\cd))-v_\bx(t,\bx_t(\cd))(\eta'_t(\cd))\\
\ss\ds=v_{\bx\bx}(t,\bx_t(\cd))(\eta_t(\cd),\eta'_t(\cd))+o(\|\eta_t\|+\|\eta'_t(\cd)\|),\qq\forall
\eta_t(\cd),\eta'_t(\cd)\in\sX_t.\ea\ee
To define the right $t$-partial derivative $v_t(t,\bx_t(\cd))$ of $v(\cd\,,\cd)$ at $(t,\bx_t(\cd))$, we need to ``fix'' $\bx_t(\cd)$ and define $v(t+\d,\bx_t(\cd))$. Naturally, we define
\bel{v(t+d)}v(t+\d,\bx_t(\cd))=v(t+\d,[\bx_t]_{t+\d}(\cd)),\qq\bx_t(\cd)\in\sX_t,\ee
where
\bel{[x]}[\bx_t]_{t+\d}(s)=\bx_t(s){\bf1}_{
[t+\d,T]}(s),\qq s\in[t+\d,T],\q\forall\bx_t(\cd)\in\sX_t.\ee
According to the above, we see that
$$[\bx_t]_{t+\d}(\cd)\in\sX_{t+\d},\qq\forall\bx_t(\cd)\in\sX_t,$$
which can be regarded as the natural ``projection'' from $\sX_t$ to $\sX_{t+\d}$.
Thus, \rf{v(t+d)} makes sense. Having such a natural restriction, we can define the right $t$-partial derivative $v_t(t,\bx_t(\cd))$ as follows:
\bel{v_t}v_t(t,\bx_t(\cd))=\lim_{\d\to 0^+}{v(t+\d,\bx_t(\cd))-v(t,\bx_t(\cd))\over\d},\ee
provided the limit exists.

\ms

To get some feeling about $v_t(t,\bx_t(\cd))$, let us present a simple example.

\begin{example}{} \rm Let
$$v(t,\bx_t(\cd))=\int_t^T\bx_t(s)^\top F(t,s)\bx_t(s)ds,\qq(t,\bx(\cd))\in\L,$$
for some nice $\dbR^{n\times n}$-valued function $F(\cd\,,\cd)$, not necessarily symmetric. Then, by \rf{v(t+d)}, one has
$$v(t+\d,\bx_t(\cd))=v(t+\d,[\bx_t]_{t+\d}(\cd))=\int_{t+\d}^T\bx_t(s)^\top F(t+\d,s)\bx_t(s)ds$$
Hence, according to \rf{v_t}, we have
$$\ba{ll}
\ss\ds v_t(t,\bx_t(\cd))=\lim_{\d\to0^+}{1\over\d}\[\int_{t+\e}^\top\3n\bx_t(s)^\top F(t+\e,s)\bx_t(s)ds-\int_t^T\3n\bx_t(s)^\top F(t,s)\bx_t(s)ds\]\\
\ss\ds\qq\qq\q=\int_t^T\3n\bx_t(s)^\top F_t(t,s)\bx_t(s)ds-\bx_t(t)^\top F(t,t)\bx_t(t),\qq\forall\bx_t(\cd)\in\sX_t.\ea$$
However, if our $v(t,\bx_t(\cd))$ is given by
\bel{v}v(t,\bx_t(\cd))=\bx_t(t)^\top F_0(t)\bx_t(t)+\int_t^T\bx_t(s)^\top F(t,s)\bx_t(s)ds,\qq(t,\bx_t(\cd))\in\L,\ee
for some nice functions $F_0(\cd)$ and $F(\cd\,,\cd)$, then
$$\ba{ll}
\ss\ds v(t+\d,\bx_t(\cd))=v(t+\d,[\bx_t]_{t+\d}(\cd))\\
\ss\ds\qq\qq\qq=\bx_t(t+\d)F_0(t+\d)\bx_t(t+\d)+\int_{t+\d}^T\bx_t(s)F(t+\d,s)\bx_t(s)ds.\ea$$
Thus, for any $(t,\bx_t(\cd))\in\L^+$, one has
$$\ba{ll}
\ss\ds v_t(t,\bx_t(\cd))=\lim_{\d\to0^+}{1\over\d}\[\bx_t(t+\d)^\top F_0(t+\d)\bx_t(t+\d)-\bx(t)^\top F_0(t)\bx(t)\\
\ss\ds\qq\qq\qq\qq+\int_{t+\e}^\top\3n\bx_t(s)^\top F(t+\e,s)\bx_t(s)ds-\int_t^T\3n\bx_t(s)^\top F(t,s)\bx_t(s)ds\]\\
\ss\ds\qq\qq~=\bx_t(t)^\top\dot F_0(t)\bx_t(t)+2\dot\bx_t(t)^\top F_0(t)\bx_t(t)+\int_t^T\3n\bx_t(s)^\top F_t(t,s)\bx_t(s)ds-\bx_t(t)^\top F(t,t)\bx_t(t),\ea$$
where $\dot{\bx}_t(t)$ stands for the right-derivative of $\bx_t(\cd)$ at $t$. In particular,
$$v_t(t,\bx)=\bx^\top\[\dot F_0(t)+\int_t^TF_t(t,s)ds-F(t,t)\]\bx,\qq\forall(t,\bx)\in\L^0.$$
From the above, we see that the function $v(\cd\,,\cd)$ given by \rf{v} does not have right $t$-partial derivative at $(t,\bx_t(\cd))\in\L\setminus\L^+$.

\end{example}

Let us introduce the following.

\begin{definition} \rm (i) Let $C^{1,2}(\L)$ be the set of all $v\in C^0(\L)$ such that $v_t,v_\bx,v_{\bx\bx}$ exist on $\L$.

\ms

(ii) Let $C^{1,2}_+(\L)$ denote the set of all $v\in C^{1,2}(\L)$ such that the following are satisfied:
\begin{enumerate}[(a)]
\item There exist constants $\k,K>0$ such that, for any $(t, \bx)\in\L$,
$$|v_t(t,\bx_t(\cd))|+\sup_{{\eta_t(\cd)\in\sX_t}\atop\|\eta_t(\cd)\|\les 1}|v_\bx(t,\bx_t(\cd))(\eta_t(\cd))|+\sup_{{\eta_t(\cd),\eta'_t(\cd)\in\sX_t}
\atop{\|\eta_t(\cd)\|, \|\eta'_t(\cd)\|\les1}}|v_{\bx\bx}(t,\bx_t(\cd))(\eta_t(\cd),\eta'_t(\cd))|\les K[1+\|\bx_t(\cd)\|^\k].$$

\item For any $\eta_t(\cd),\eta'_t(\cd)\in\sX_0$, $v_t(t,\bx_t(\cd)), v_\bx(t,\bx_t(\cd))(\eta_t(\cd)),v_{\bx\bx}(t,\bx_t(\cd))(\eta_t(\cd),\eta'_t(\cd))$ are continuous in $(t,\bx_t(\cd))$, where the continuity in $t$ always means right-continuity.

\item There exist $\k>0$ and a modulus of continuity $\rho$ such that:
$$\ba{ll}
\ss\ds\big|[v_{\bx\bx}(t,\bx_t(\cd))-v_{\bx\bx}(t,\bx'_t(\cd))](\eta_t(\cd),\eta_t(\cd))\big| \les\big[1+\|\bx_t(\cd)\|^\k+\|\bx'_t(\cd)\|^\k\big]\|\eta_t(\cd)\|^2 \rho(\|\bx_t(\cd)-\bx'_t(\cd)\|),\\
\ss\ds\qq\qq\qq\qq\qq\qq\qq\qq\qq\forall(t,\bx_t(\cd))\in\L,~\eta_t(\cd),\eta'_t(\cd)\in\sX_t.\ea$$
\end{enumerate}
\end{definition}

Note that for $v\in C^{1,2}_+(\L)$, one has
\bel{Taylor}\ba{ll}
\ss\ds v(t,\bx_t(\cd)+\eta_t(\cd))-v(t,\bx_t(\cd))=v_\bx(t,\bx_t(\cd))\eta_t(\cd)
+{1\over2}v_{\bx\bx}(t,\bx_t(\cd))(\eta_t(\cd),\eta_t(\cd))+o(\|\eta_t(\cd)\|^2),\\
\ss\ds\qq\qq\qq\qq\qq\qq\qq\qq\qq\qq\qq\qq\forall\bx_t(\cd)\eta_t(\cd)\in\L.\ea\ee

The following is essentially a version of the functional It\^o's formula for SVIEs found in Viens--Zhang \cite{Viens-Zhang2019}.
For completeness, we sketch the proof.

\begin{proposition}\label{Ito} \sl Let $b,\si:\D_*[0,T]\times\Om\to\dbR^n$ be measurable such that $s\mapsto(b(s,\t),\si(s,\t))$ is differential, $\t\mapsto(b(s,\t),\si(s,\t))$ is $\dbF^t$-progressively measurable on $[t,T]$ and
$$\dbE\[\sup_{s\in[t,T]}\int_t^T\(|b(s,\t)|+|\si(s,\t)|^2\)d\t\]<\infty.$$
Let
\bel{cX}\cX(s,r)=\bx_t(s)+\int_t^rb(s,\t)d\t+\int_t^r\si(s,\t)dW(\t),\qq t\les r\les s\les T.\ee
and $v(\cd\,,\cd)\in C_+^{1,2}(\L)$. Then the following functional It\^{o}'s formula holds:
\bel{Ito-formula}\ba{ll}
\ss\ds dv(r,\cX(\cd\,,r))=\[v_t(r,\cX(\cd\,,r))
+{1\over2}v_{\bx\bx}(r,\cX(\cd\,,r))\big(\si(\cd\,,r),\si(\cd\,,r)\big) +v_\bx(r,\cX(\cd\,,r))\big(b(\cd\,,r)\big)\]dr\\
\ss\ds\qq\qq\qq\qq\qq+v_\bx(r,\cX(\cd\,,r))\big(\si(\cd\,,r)\big)dW(r).\ea\ee
\end{proposition}

\begin{proof} Let $\Pi:r=r_0<r_1<\cds<r_N=T$ be a partition of $[r,T]$ with mesh size
$$\|\Pi\|=\max_{0\les i\les N-1}|r_{i+1}-r_i|.$$
We have
\bel{Ito-Q-Proof1}\ba{ll}
\ss\ds v\big(T,\cX(\cd\,,T)\big)-v\big(r,\cX(\cd\,,r)\big)=v\big(r_N,\cX(\cd\,,r_N)\big)-
v\big(r_0,\cX(\cd\,,r_0)\big)=\sum_{i=1}^N[I_i^1+I_i^2],\ea\ee
where
\bel{I_i^1}\ba{ll}
\ss\ds I_i^1=v\big(r_{i+1},\cX(\cd\,,r_i)\big)-v\big(r_i,\cX(\cd\,,r_i)\big),\\
\ss\ds I_i^2=v\big(r_{i+1},\cX(\cd\,,r_{i+1})\big)-v\big(r_{i+1},\cX(\cd\,,r_i)\big).\ea\ee
By \rf{v(t+d)} and \rf{v_t}, we have
\bel{I-1}\lim_{\|\Pi\|\to0}\sum_{i\ges1}I_i^1=\lim_{\|\Pi\|\to0}\sum_{i\ges1}\int_{s_i}^{s_{i+1}}v_t\big(\t,\cX(\cd\,,s_i)\big)d\t=\int_r^Tv_t\big(\t,\cX(\cd\,,\t)\big)d\t.\ee
Next,
$$\ba{ll}
\ss\ds I_i^2=v\big(r_{i+1},\cX(\cd\,,r_{i+1})\big)-v\big(r_{i+1},\cX(\cd\,,r_i)\big)
=v\big(r_{i+1},\cX(\cd\,,r_{i+1})\big)-v\big(r_{i+1},[\cX(\cd\,,r_i)]_{r_{i+1}}\big)\\
\ss\ds\q=v_\bx\big(r_{i+1},[\cX(\cd\,,r_i)]_{r_{i+1}}\big)\big(\cX(\cd\,,r_{i+1}\big)-[\cX(\cd\,,
r_i)]_{r_{i+1}}\big)\\
\ss\ds\qq+{1\over2}v_{\bx\bx}\big(r_{i+1},[\cX(\cd\,,r_i)]_{r_{i+1}}\big)\(\big(\cX(\cd\,,r_{i+1}\big)-[\cX(\cd\,,
r_i)]_{r_{i+1}}\big),\big(\cX(\cd\,,r_{i+1}\big)-[\cX(\cd\,,
r_i)]_{r_{i+1}}\big)\)\\
\ss\ds\qq+o\(\|\cX(\cd\,,r_{i+1}\big)-[\cX(\cd\,,r_i)]_{r_{i+1}}\|^2\).\ea$$
By \rf{cX}, we have
\bel{bi-sii}\cX(\th,r_{i+1})-[\cX(\th,r_i)]_{r_{i+1}}=\int_{r_i}^{r_{i+1}}b(\th,\t)d\t
+\int_{r_i}^{r_{i+1}}\si(\th,\t)dW(\t),\q\forall\th\in[s_{i+1},T].\ee
Thus,
$$\ba{ll}
\ss\ds v_\bx\big(r_{i+1},[\cX(\cd\,,r_i)]_{r_{i+1}}\big)\big(\cX(\cd\,,r_{i+1}\big)-[\cX(\cd\,,
r_i)]_{r_{i+1}}\big)\\
\ss\ds=v_\bx\big(r_{i+1},[\cX(\cd\,,r_i)]_{r_{i+1}}\big)\(\int_{r_i}^{r_{i+1}}b(\cd\,,\t)dt
+\int_{r_i}^{r_{i+1}}\si(\cd\,,\t)dW(\t)\)\\
\ss\ds=\int_{r_i}^{r_{i+1}}v_\bx\big(r_{i+1},[\cX(\cd\,,r_i)]_{r_{i+1}}\big)b(\cd\,,\t)dt
+\int_{r_i}^{r_{i+1}}v_\bx\big(r_{i+1},[\cX(\cd\,,r_i)]_{r_{i+1}}\big)\si(\cd\,,\t)dW(\t)\ea$$
The term involving $v_{\bx\bx}$ can be handled similarly. Hence,
$$\ba{ll}
\ss\ds\lim_{\|\Pi\|\to0}\sum_{i\ges0}I_i^2=\int_t^r\[v_\bx\big(\t,\cX(\cd\,,\t)\big)b(\cd\,,\t)
+{1\over2}v_{\bx\bx}\big(t,\cX(\cd\,,r\big)\big(\si(\cd\,,\t),\si(\cd\,,\t)\big)\]d\t\\
\ss\ds\qq\qq\qq\qq+\int_t^rv_\bx(r,\cX_t(\cd\,,\t))\cX(\cd\,,\t)dW(\t).\ea$$
This proves our conclusion. \end{proof}

%\begin{remark}
%In Viens--Zhang \cite{Viens-Zhang2019} and Wang--Yong--Zhang \cite{Wang-Yong-Zhang2021}, the It\^o's formula was derived for any %function $v(\cd\,,\cd)$ in a larger space than $\L$. Our functions are defined in a smaller space.
%\end{remark}

\ms
Next, we denote by $\cS^n_t\deq\cS(\sX_t)$, the set of all bounded symmetric bilinear functional on $\sX_t\times\sX_t$, and let $C([0,T];\cS^n)$ be the set of all $P(\cd)$ having
property that
$$P(t)\in\cS^n_t,\q\forall t\in[0,T],$$
and the map $t\mapsto P(t)$ is continuous in the following sense:
$$\lim_{t\to t_0}P(t)(\bx_t,\bx_t)=P(t_0)([\bx_t]_{t_0},[\bx_t]_{t_0})=:P(t_0)(\bx_{t_0},\bx_{t_0}).$$
with respect to the norm $\|\cd\|_\cS$. Further, we let $C^1([0,T];\cS^n)$ be the set of all $P(\cd)\in C([0,T];\cS^n)$ such that
\bel{v=P(t)}v(t,\bx_t(\cd))=P(t)(\bx_t(\cd),\bx_t(\cd))\equiv\bx_t(\cd)^\top P(t)\bx_t(\cd),\qq\forall(t,\bx_t(\cd))\in\L,\ee
admits
$$v_t(t,\bx_t(\cd))\equiv\dot P(t)\big(\bx_t(\cd),\bx_t(\cd)\big),$$
which is continuous. Note that for $t\in[0,T]$, if stochastic processes $X(\cd)$ and $Y(\cd)$ are in $\sX_t$ almost surely, i.e., $X(\cd)$ and $Y(\cd)$ have continuous paths on $[t,T]$,
then
$$Y(\cd)^\top P(t)X(\cd)\equiv P(t)\big(X(\cd),Y(\cd)\big)$$
is well-defined almost surely. We have the following result.
\begin{proposition}\label{Ito-RE} \sl Let the assumptions of Proposition \ref{Ito} hold. Let $\cX(\cd,\cd)$ be defined by \rf{cX}. Then the following functional It\^{o}'s formula holds:
\begin{align}
 d\big[P(r)\big([\bx_s'(\cd)]_r,\cX(\cd\,,r)\big)\big]
&=\big[\dot P(r)\big([\bx_s'(\cd)]_r,\cX(\cd\,,r)\big)\1n+\1n P(r)\big([\bx_s'(\cd)]_r,b(\cd\,,r)\big)
\big]dr\nn\\
&\q+P(r)\big([\bx_s'(\cd)]_r,\si(\cd\,,r)\big)dW(r),\qq r\in[s,T],~\bx_s'(\cd)\in\sX_s,\nn\\
 d\big[P(r)\big(\cX(\cd\,,r),\cX(\cd\,,r)\big)\big]
&=\big[\dot P(r)\big(\cX(\cd\,,r),\cX(\cd\,,r)\big)\1n+\1n 2P(r)\big(\cX(\cd\,,r),b(\cd\,,r)\big)
+\1n P(r)\big(\si(\cd\,,r),\si(\cd\,,r)\big)
\big]dr\nn\\
&\q+2P(r)\big(\cX(\cd\,,r),\si(\cd\,,r)\big)dW(r),\qq r\in[t,T].
\end{align}

\end{proposition}

The proof is a direct consequence  of Proposition \ref{Ito}. We omit it here.

\section{Solvability of Problem (LQ-FSVIE) and Optimality System}\label{sect-opti-system}

In this section, we shall study the functional \rf{cost} as a quadratic functional of the controls $u(\cd)$ on the Hilbert space $\sU[t,T]$. A necessary condition and a sufficient condition for the existence of an open-loop optimal control will be derived by a standard variational method. This will be expressed by an optimality system which is an FBSVIE.

\ms

For any $u(\cd)\in\sU[t,T]$, consider the following FSVIE:
\bel{state-0-u}\ba{ll}
\ds X^{0,u}(s)=\int_t^s\big[A(s,r)X^{0,u}(r)+B(s,r)u(r)\big]dr+ \int_t^s\big[C(s,r)X^{0,u}(r)+D(s,r)u(r)\big]dW(r),\\
\ds\qq\qq\qq\qq\qq\qq\qq\qq\qq\qq\qq\qq\qq\qq\qq s\in[t,T].\ea\ee
By \autoref{lemma:Well-SVIE}, the above FSVIE admits a unique solution $X^{0,u}(\cd)\in L_{\dbF^t}^2(\Om;C([t,T];\dbR^n))$
satisfying
\bel{K}\dbE\[\sup_{s\in[t,T]}|X^{0,u}(s)|^2\]\les K\dbE\int_t^T|u(s)|^2ds,\ee
where the constant $K>0$ is independent of $u(\cd)$.
Thus we can define two bounded linear operators $\G:\sU[t,T]\to L_{\dbF^t}^2(\Om;C([t,T];\dbR^n))\subseteq L_{\dbF^t}^2(t,T;\dbR^n)$
and $\h\G:\sU[t,T]\to L_{\cF^t_T}^2(\Om;\dbR^n)$ as follows:
\bel{def-G-hG}[\G u(\cd)](\cd)=X^{0,u}(\cd),\q[\h\G u(\cd)]=X^{0,u}(T),\qq\forall u(\cd)\in\sU[t,T].\ee
For any $\bx_t(\cd)\in C([t,T];\dbR^n)$, let $X^{\bx_t,0}(\cd)$ be the unique solution to the following linear un-controlled FSVIE:
\bel{state-f-0}X^{\bx_t,0}(s)=\bx_t(s)+\int_t^s A(s,r)X^{\bx_t,0}(r)dr+\int_t^sC(s,r)X^{\bx_t,0}(r)dW(r), \q s\in[t,T].\ee
Then the following linear operators $\varXi:C([t,T];\dbR^n)\to L_{\dbF^t}^2(\Om;C([t,T];\dbR^n))$ and $\h\varXi:C([t,T];\dbR^n)\to L_{\cF^t_T}^2(\Om;\dbR^n)$ can also be well-defined:
\bel{def-Xi-hXi}
[\varXi\bx_t(\cd)](\cd)=X^{\bx_t,0}(\cd),\q[\h\varXi\bx_t(\cd)]
=X^{\bx_t,0}(T),\qq\forall\bx_t(\cd)\in C([t,T];\dbR^n).\ee
With \rf{def-G-hG} and \rf{def-Xi-hXi}, the unique solution $X(\cd)\equiv X(\cd\,;t,\bx_t(\cd),u(\cd))$ of state equation \rf{state} corresponding to $(t,\bx_t(\cd))\in\L$ and $u(\cd)\in\sU[t,T]$ can be represented by
\bel{state-functional-representation}
X(\cd)=[\G u(\cd)](\cd)+[\varXi\bx_t(\cd)](\cd),\q X(T)=[\h\G u(\cd)]+[\h\varXi\bx_t(\cd)].\ee
By substituting the above into \rf{cost}, we obtain the following representation of the functional \rf{cost}:
\bel{cost-functional-rewrite}
J(t,\bx_t(\cd);u(\cd))=\lan\cM_2u,u\ran+2\lan\cM_1\bx_t,u\ran
+\lan\cM_0\bx_t,\bx_t\ran,\ee
where
\bel{M}\cM_2\deq {\G^*Q\G+\h\G^*G\h\G+R\over 2}\qq \cM_1\deq{\G^*Q\varXi+\h\G^*G\h\varXi\over 2},\qq\cM_0={\varXi^*Q\varXi+\h\varXi^*G\h\varXi\over2}.\ee
Using the representation \rf{cost-functional-rewrite},
we get the following abstract characterization for the open-loop optimal controls of Problem (LQ-FSVIE).

\begin{proposition}\label{prop:suffi-nece-condition} Let $(t,\bx_t(\cd))\in\L$ be any given free pair and $\bar u(\cd)\in\sU[t,T]$. Then $\bar u(\cd)$ is an open-loop optimal control of {\rm Problem (LQ-FSVIE)} for $(t,\bx_t(\cd))$ if and only if
\bel{main-suffi-nece-condition}\cM_2\ges0,\qq\hbox{and}\qq\cM_2\bar u+\cM_1\bx_t=0.\ee
\end{proposition}

\begin{proof} It is clear to see that $\bar u(\cd)$ is an optimal control of  Problem (LQ-FSVIE) if and only if
\bel{suffi-nece-condition1}
J(t,\bx_t(\cd);\bar u(\cd)+\l u(\cd))-J(t,\bx_t(\cd);\bar u(\cd))\ges0,\qq\forall u(\cd)\in\sU[t,T],~\l\in\dbR.\ee
For any $u(\cd)\in\sU[t,T]$ and $\l\in\dbR$, by \rf{cost-functional-rewrite}, we have
\bel{suffi-nece-condition2}
J(t,\bx_t(\cd);\bar u(\cd)+\l u(\cd))-J(t,\bx_t(\cd);\bar u(\cd))=\l^2\lan\cM_2u,u\ran+2\l\lan\cM_2\bar u+\cM_1\bx_t,u\ran.\ee
Thus \rf{suffi-nece-condition1} holds if and only if \rf{main-suffi-nece-condition} holds. The proof is thus complete. \end{proof}

To solve Problem (LQ-FSVIE), we introduce the following assumption.

\begin{taggedassumption}{(H3)}\label{ass:H3}\rm
There exists a constant $\a>0$ such that
\bel{uniformly-convex}
\lan\cM_2u,u\ran=J(t,0;u(\cd))\ges \alpha\dbE\int_t^T|u(s)|^2ds,\q\forall u(\cd)\in\sU[t,T].\ee
\end{taggedassumption}

Combining this condition with \autoref{prop:suffi-nece-condition}, we can obatin the following result easily.

\begin{corollary}\label{corollary:suffi-condition}
Let {\rm \ref{ass:H1}--\ref{ass:H3}} hold. Then, for any free pair $(t,\bx_t(\cd))\in\L$, {\rm Problem (LQ-FSVIE)} admits a unique optimal control $\bar u(\cd)$, which is given by
\bel{corollary:suffi-condition-main}
\bar u(\cd)=-(\cM_2^{-1}\cM_1\bx_t)(\cd).\ee
\end{corollary}

Combining \autoref{prop:suffi-nece-condition} and \cite[Theorem 5.2]{Yong2008}, we have the following results.

\begin{theorem}\label{theorem:optimality-condition}
Suppose the following convexity condition holds:
\bel{convex}\lan\cM_2u,u\ran=J(t,0;u(\cd))\ges 0,\qq\forall u(\cd)\in\sU[t,T].\ee
The control $\bar u(\cd)\in\sU[t,T]$ is an optimal control of {\rm Problem (LQ-FSVIE)} if and only if the optimality system \rf{MP}--\rf{OP0} hold.
\end{theorem}

\begin{proof} From \rf{suffi-nece-condition2} we see that
$$\lim_{\l\to 0}{J(t,\bx_t(\cd);\bar u(\cd)+\l u(\cd))-J(t,\bx_t(\cd);\bar u(\cd))\over\l}=2\lan\cM_2\bar u+\cM_1\bx_t,u\ran.$$
On the other hand, by \cite[Theorem 5.2]{Yong2008} we have
$$\lim_{\l\to 0}{J(t,\bx_t(\cd);\bar u(\cd)+\l u(\cd))-J(t,\bx_t(\cd);\bar u(\cd))\over\l}=\dbE\int_t^T\lan R(s)\bar u(s)+Y^0(s),u(s)\ran ds.$$
Thus,
$$2(\cM_2\bar u+\cM_1\bx_t)(\cd)=(R\bar u+Y^0)(\cd).$$
Then the necessity follows from \cite[Theorem 5.2]{Yong2008}, and sufficiency follows from \autoref{prop:suffi-nece-condition}.
\end{proof}
\begin{remark}\rm\label{remark3} Let the controlled system reduce to an SDE. Then the optimality condition \rf{MP} implies (noting the equation for $Y^0(\cd)$ in \rf{OP0})
\bel{OC-SDE}\ba{ll}
\ss\ds0=R(s)\bar u(s)+B(s)^\top G\dbE_s[\bar X(T)]+D(s)^\top\zeta(s)+B(s)^\top \dbE_s\int_s^TY(r)dr+D(s)^\top\int_s^T Z(r,s)dr\\
\ss\ds\q=R(s)\bar u(s)+B(s)^\top\dbE_s\(G\bar X(T)+\int_s^TY(r)dr\)+D(s)^\top\(\z(s)+\int_s^T Z(r,s)dr\),\q s\in[t,T].\ea\ee
From the second equation in \rf{OP0}, one has
$$G\bar X(T)=\dbE_s[G\bar X(T)]+\int_s^T\z(\t)dW(\t).$$
Since $(Y(\cd),Z(\cd\,,\cd))$ is the M-solution to the third equation in \rf{OP0}, we have
$$Y(r)=\dbE_s[Y(r)]+\int_s^rZ(r,\t)dW(\t).$$
Hence,
\bel{SDE-OP-1}\ba{ll}
\ss\ds G\bar X(T)+\int_s^T Y(r)dr=\dbE_s[G\bar X(T)]+\int_s^T\z(\t)dW(\t)+\int_s^T\[\dbE_s[Y(r)]+\int_s^rZ(r,\t)
dW(\t)\]dr\\
\ss\ds\qq\qq\qq\qq\q~=\dbE_s[G\bar X(T)]+\dbE_s\int_s^TY(r)dr+\int_s^T\[\z(\t)+\int_\t^T Z(r,\t)dr\]dW(\t),\ea\ee
and by the third equation in \rf{OP0},
$$\ba{ll}
\ss\ds\int_s^T\3n Y(r)dr=\2n\int_s^T\3n\dbE_r\Big\{Q(r)\bar X(r)+A(r)^\top G\bar X(T)+C(r)^\top\z(r)+\2n\int_r^T\3n\big[A(r)^\top Y(\t)+C(r)^\top Z(\t,r)\big]d\t\Big\}dr\\
\ss\ds\qq\q=\int_s^T\3n\Big\{Q(r)\bar X(r)\1n+\1n A(r)^\top\dbE_r[G\bar X(T)]\1n+\1n A(r)^\top\dbE_r\2n\int_r^T\3n Y(\t)d\t\1n+\1n C(r)^\top\z(r)\1n+\1n C(r)^\top\3n\int_r^T\3n Z(\t,r)d\t\Big\}dr.\ea$$
Substituting the above into \rf{SDE-OP-1} yields that
$$\ba{ll}
\ss\ds\dbE_s[G\bar X(T)]+\dbE_s\int_s^TY(r)dr-\int_s^T\[\zeta(r)+\int_r^T Z(\t,r)d\t\]dW(r)\\
\ss\ds=G\bar X(T)+\2n\int_s^T\2n\[\{Q(r)\bar X(r)+A(r)^\top\(\dbE_r[ G\bar X(T)]+\dbE_r\2n\int_r^T\2n Y(\t)d\t\)+C(r)^\top\(\z(r)+\2n\int_r^T\2n Z(\t,r)d\t\)\]dr.\ea$$
If we denote
$$p(s)=\dbE_s\(G\bar X(T)+\int_s^TY(r)dr\),\q q(s)=\zeta(s)+\int_s^TZ(r,s)dr,\q s\in[t,T],$$
then $(p(\cd),q(\cd))$ satisfies the following BSDE:
$$p(s)=G\bar X(T)+\int_s^T\big[ A(r)^\top p(r)+C(r)^\top q(r)+Q(r)\bar X(r)\big]dr-\int_s^T q(r)dW(r),\q s\in[t,T],
$$
and the  optimality condition \rf{OC-SDE} can be rewritten as
$$
R(s)\bar u(s)+B(s)^\top p(s)+D(s)^\top q(s)=0,\q s\in[t,T],
$$
which recovers the corresponding results of Problem (LQ-SDE) (see \cite[Chapter 6]{Yong-Zhou1999}, for example).
\end{remark}

%$$Y^0(s)=B(T,s)^\top G\bar X(T)+D(T,s)^\top\zeta(s)+\int_s^T\big[B(r,s)^\top Y(r) +D(r,s)^\top Z(r,s)\big]dr-\int_s^T  Z^0(s,r)dW(r).$$

\ms

Let us now return to BSVIEs. Denote
\bel{psi}\ba{ll}
\ss\ds\psi(s)=Q(s)\bar X(s)+A(T,s)^\top G\bar X(T)+C(T,s)^{\top}\z(s),\\ [2mm]
\ss\ds\psi^0(s)=B(T,s)^\top G\bar X(T)+D(T,s)^\top\z(s),\ea \qq\q s\in[t,T].\ee
Then \rf{OP0} can written as (assuming that $R(s)^{-1}$ exists)
\bel{OP}\left\{\1n\ba{ll}
\ds\bar X(s)=\bx_t(s)+\int_t^s\big[A(s,\t)\bar X(\t)-B(s,\t)R(\t)^{-1}Y^0(\t)\big]d\t\\
\ss\ds\qq\qq+\int_t^s\big[C(s,\t)\bar X(\t)-D(s,\t)R(\t)^{-1}Y^0(\t)\big]dW(\t),\\
\ss\ds\eta(s)=G\bar X(T)-\int_s^T\z(\t)dW(\t),\\
\ds Y(s)=\psi(s)+\int_s^T\big[A(\t,s)^\top Y(\t)+C(\t,s)^\top Z(\t,s)\big]d\t-\int_s^T Z(s,\t)dW(\t),\\
\ds Y^0(s)=\psi^0(s)+\int_s^T\big[B(\t,s)^\top Y(\t)+D(\t,s)^\top Z(\t,s)\big]d\t-\int_s^TZ^0(s,\t)dW(\t).\ea\right.\ee
Note that \rf{OP} is a coupled system of FBSVIEs. To our best knowledge, there is no general result on the solvability of coupled FBSVIEs on an arbitrary time horizon. In the rest of the paper, we are going to develop a  decoupling method, which is the most important contribution of our paper. The key point is to find the so-called {\it decoupling field} for FBSVIE \rf{OP}. As a preparation, we next provide a new representation for the optimality condition of Problem (LQ-FSVIE).

\ms

We introduce the following system of BSVIEs on $[t,T]$:
\bel{Type-III}\left\{\1n\ba{ll}
\ds Y^A(s)=\int_s^T A(\t,s)^\top\[\dbE_\t[\psi(\t)]+Y^A(\t)+ Z^C(\t,\t)\]d\t-\int_s^TZ^A(s,\t)dW(\t),\\
\ss\ds Y^B(s)=\int_s^TB(\t,s)^\top\[\dbE_\t[\psi(\t)]+Y^A(\t)+Z^C(\t,\t)\]dr
-\int_s^TZ^B(s,\t)dW(r),\\
\ss\ds Y^C(s)=\int_s^TC(\t,s)^\top\[\dbE_\t[\psi(\t)]+Y^A(\t)+Z^C(\t,\t)\]d\t-\int_s^TZ^C(s,\t)
dW(\t),\\
\ss\ds Y^D(s)=\int_s^TD(\t,s)^\top\[\dbE_\t[\psi(\t)]+Y^A(\t)+Z^C(\t,\t)\]d\t
-\int_s^TZ^D(s,\t)dW(\t).\ea\right.\ee
Note that in the above, the diagonal value $Z^C(\t,\t)$ of the process $Z^C(\cd\,,\cd)$ appears, which makes such BSVIEs essentially different from Type-I and Type-II BSVIEs. Thus, we call the above Type-III BSVIEs to distinguish them from the others. BSVIEs with the diagonal values $Z(\t,\t)$ presented were introduced by Wang and Yong \cite{Wang-Yong2021} the first time while studying the time-inconsistent optimal control problems for FSVIEs. The well-posedness of such type of BSVIEs with general generators was established by Hern\'{a}ndez and Possamai \cite{HP2}. Here, it is the first time that we use the name ``Type-III'' for these types of BSVIEs. In \cite{Wang-Yong2021,HP2}, the diagonal value $Z(\t,\t)$ was due to the equilibrium HJB equations. Here, we will use this type of BSVIEs to represent the solutions of adjoint equations in the optimality system. This is surprising.

\begin{proposition}\label{Prop:Type-I-II} Let {\rm\ref{ass:H1}} hold. Then for any $\psi(\cd)\in L^2_{\cF^t_T}([t,T];\dbR^n)$ and $\psi^0(\cd)\in L^2_{\cF^t_T}([t,T];\dbR)$, the third and the fourth BSVIEs in \rf{OP} have unique adapted M-solution $(Y(\cd),Y^0(\cd),Z(\cd\,,\cd),Z^0(\cd\,,\cd))$, and Type-III BSVIE \rf{Type-III} also admits a unique adapted solution $(Y^*(\cd),Z^*(\cd\,,\cd))$ with $*=A,B,C,D$. Moreover, the following representation holds:
\bel{Prop:Type-I-II-main2}\left\{\1n\ba{ll}
\ss\ds Y(s)=\dbE_s[\psi(s)]+Y^A(s)+Z^C(s,s),\\
\ss\ds Y^0(s)=\dbE_s[\psi^0(s)]+Y^B(s)+Z^D(s,s),\ea\right.\qq s\in[t,T].\ee

\end{proposition}

\begin{proof} Let $(Y(\cd),Z(\cd\,,\cd))$ be the adapted M-solution to the third equation in \rf{OP}. Using the fact
$$Y(s)=\dbE_r[Y(s)]+\int_r^sZ(s,\t)dW(\t),\qq t\les r\les s\les T,$$
we get
\bel{Prop:Type-I-II-proof1}\ba{ll}
\ss\ds\int_r^TA(s,r)^\top Y(s)ds=\int_r^TA(s,r)^\top\dbE_r[Y(s)]ds+\int_r^TA(s,r)^\top\int_t^sZ(s,\t)dW(\t)ds\\ [2mm]
\ss\ds\qq\qq\qq\qq~=\dbE_r\int_r^TA(s,r)^\top Y(s)ds+\int_r^T\int_s^TA(\t,r)^\top Z(\t,s)d\t dW(s).\ea\ee
On the other hand, by the third equation in \rf{OP},
$$Y(s)=\dbE_s\Big\{\psi(s)+\int_s^T\big[A(\t,s)^\top Y(\t)+C(\t,s)^\top Z(\t,s)\big]d\t\Big\}.$$
Thus, we have
$$\ba{ll}
\ss\ds\int_r^TA(s,r)^\top Y(s)ds=\int_r^T A(s,r)^\top\dbE_s\Big\{\psi(s)+\int_s^T\big[A(\t,s)^\top Y(\t)+C(\t,s)^\top Z(\t,s)\big]d\t\Big\}ds\\
\ss\ds=\1n\int_r^T\2n A(s,r)^\top\dbE_s[\psi(s)]ds+\2n\int_r^T\2n A(s,r)^\top\dbE_s\2n\int_s^T\2n A(\t,s)^\top Y(\t)d\t ds+\2n\int_r^T\2n A(s,r)^\top\2n\int_s^T\2n C(\t,s)^\top Z(\t,s)d\t ds,\ea$$
where the last equality is due to the fact that $Z(\t,s)$ is $\cF^t_s$-measurable. By substituting the above into \rf{Prop:Type-I-II-proof1}, we get
%
%\bel{Prop:Type-I-II-proof2}
$$\ba{ll}
\ss\ds\dbE_r\1n\int_r^T\2n A(s,r)^\top Y(s)ds=\2n\int_r^T\3n A(s,r)^\top\[\dbE_s[\psi(s)]+\1n\dbE_s\2n\int_s^T\2n A(\t,s)^\top Y(\t)d\t+\1n\int_s^T\2n C(\t,s)^\top Z(\t,s)d\t\]ds\\
\ss\ds\qq\qq\qq\qq\qq-\int_r^T\int_s^TA(\t,r)^\top Z(\t,s)d\t dW(s).\ea$$%\ee
Similarly, we have
$$\ba{ll}
\ss\ds\dbE_r\1n\int_r^T\2n B(s,r)^\top Y(s)ds=\2n\int_r^T\3n B(s,r)^\top\[\dbE_s[\psi(s)]+\1n\dbE_s\2n\int_s^T\2n A(\t,s)^\top Y(\t)d\t+\1n\int_s^T\2n C(\t,s)^\top Z(\t,s)d\t\]ds\\
\ss\ds\qq\qq\qq\qq\qq-\int_r^T\int_s^TB(\t,r)^\top Z(\t,s)d\t dW(s),\ea$$%\ee
$$\ba{ll}
\ss\ds\dbE_r\1n\int_r^T\2n C(s,r)^\top Y(s)ds=\2n\int_r^T\3n C(s,r)^\top\[\dbE_s[\psi(s)]+\1n\dbE_s\2n\int_s^T\2n A(\t,s)^\top Y(\t)d\t+\1n\int_s^T\2n C(\t,s)^\top Z(\t,s)d\t\]ds\\
\ss\ds\qq\qq\qq\qq\qq-\int_r^T\int_s^TC(\t,r)^\top Z(\t,s)d\t dW(s),\ea$$%\ee
and
$$\ba{ll}
\ss\ds\dbE_r\1n\int_r^T\2n D(s,r)^\top Y(s)ds=\2n\int_r^T\3n D(s,r)^\top\[\dbE_s[\psi(s)]+\1n\dbE_s\2n\int_s^T\2n A(\t,s)^\top Y(\t)d\t+\1n\int_s^T\2n C(\t,s)^\top Z(\t,s)d\t\]ds\\
\ss\ds\qq\qq\qq\qq\qq-\int_r^T\int_s^TD(\t,r)^\top Z(\t,s)d\t dW(s).\ea$$%\ee
Thus, if we denote
\bel{Prop:Type-I-II-proof5}\ba{ll}
\ss\ds Y^A(s)=\dbE_s\int_s^TA(\t,s)^\top Y(\t)d\t,\q Z^A(r,s)=\int_s^TA(\t,r)^\top Z(\t,s)d\t,\\
\ss\ds Y^B(s)=\dbE_s\int_\t^TB(\t,s)^\top Y(\t)d\t,\q Z^B(r,s)=\int_s^TB(\t,r)^\top Z(\t,s)d\t,\\
\ss\ds Y^C(s)=\dbE_s\int_\t^TC(\t,s)^\top Y(\t)ds,\q Z^C(r,s)=\int_s^TC(\t,r)^\top Z(\t,s)d\t,\\
\ss\ds Y^D(s)=\dbE_s\int_\t^TD(\t,s)^\top Y(\t)d\t,\q Z^D(r,s)=\int_s^TD(\t,r)^\top Z(\t,s)d\t,\ea\ee
then \rf{Type-III} is satisfied. Thus, the process $(Y^A(\cd),Y^B(\cd),Y^C(\cd),Y^D(\cd),Z^A(\cd\,,\cd),Z^B(\cd\,,\cd),
Z^C(\cd\,,\cd),Z^D(\cd\,,\cd))$ defined by \rf{Prop:Type-I-II-proof5} is the unique adapted solution to \rf{Type-III}. Applying $\dbE_s$ on the last two equations in \rf{OP} yields
$$\ba{ll}
\ss\ds Y(s)=\dbE_s\[\psi(s)+\int_s^T\(A(\t,s)^\top Y(\t)+C(\t,s)^\top Z(\t,s)\)d\t\]=\dbE_s[\psi(s)]+Y^A(s)+Z^C(s,s),\\
\ss\ds Y^0(s)=\dbE_s\[\psi^0(s)+\int_s^T\(B(\t,s)^\top Y(\t)+D(\t,s)^\top Z(\t,s)\)d\t\]=\dbE_s[\psi^0(s)]+Y^B(s)+Z^D(s,s),\ea$$
proving our conclusion. \end{proof}

\begin{remark} \autoref{Prop:Type-I-II} gives an explicit relation between  a Type-II BSVIE and a Type-III BSVIE. This relation will serve as a foundation in developing our decoupling approach for the optimality system associated Problem (LQ-FSVIE).
\end{remark}

\begin{theorem}\label{theorem:optimality-condition-new}
Let {\rm\ref{ass:H1}--\ref{ass:H2}} hold. Suppose that the convexity condition \rf{convex} holds. Let $\psi(\cd)$ and $\psi^0(\cd)$ be defined by \rf{psi}. Let $(Y^A(\cd),Y^B(\cd),Y^C(\cd),Y^D(\cd),Z^A(\cd\,,\cd),Z^B(\cd\,,\cd),
Z^C(\cd\,,\cd),Z^D(\cd\,,\cd))$ be the unique adapted solution to BSVIE \rf{Type-III}. Then the control $\bar u(\cd)\in\sU[t,T]$ is an open-loop optimal control of {\rm Problem (LQ-FSVIE)} if and only if
\bel{theorem:optimality-condition-main1-rw}
R(s)\bar u(s)+\dbE_s[\psi^0(s)]+Y^B(s)+Z^D(s,s)=0,\qq s\in[t,T],~\as
\ee
\end{theorem}

\section{Derivation of the Path-dependent Riccati Equation}\label{sect-Riccati}

In this section, we will find the path-dependent Riccait equation for our Problem (LQ-FSVIE) via the HJB equation for the value function. The procedure is formal. However, the arguments over verification is rigorous. Thus, once the well-posedness of the Riccati equation is established, our Problem (LQ-FSVIE) is solved. This also decouples the optimality system \rf{OP}.

\ms

%\subsection{Path-dependent HJB equations}

By \cite[Subsection 4.3]{Viens-Zhang2019}, the path-dependent HJB equation associated with Problem (LQ-FSVIE) reads
\bel{HJB-PPDE1}\left\{\2n\ba{ll}
\ss\ds v_t(t,\bx)+\inf_{u\in\dbR^m}\cH(t,\bx,u,v_\bx (t,\bx),v_{\bx\bx}(t,\bx))=0,\qq(t,\bx)\in\L,\\
\ss\ds v(T,\bx)={1\over2}\lan G\bx(T),\bx(T)\ran,\ea\right.\ee
where the Hamiltonian $\cH$ is defined by
\bel{Hamliton}\ba{ll}
\ss\ds\cH(t,\bx,u,v_\bx(t,\bx),v_{\bx\bx}(t,\bx))\deq{1\over 2} \big[C(\cd\,,t)\bx(t)+D(\cd\,,t)u\big]^\top v_{\bx\bx}(t,\bx)\big[C(\cd\,,t)\bx(t)
+D(\cd\,,t) u\big)\big]\\
\ss\ds\qq\qq\qq\qq\qq\qq\q+v_\bx(t,\bx)\big[A(\cd\,,t)
\bx(t)\1n+\1n B(\cd\,,t)u\big]\1n+
\1n{1\over2}\bx(t)^\top Q(t)\bx(t)\1n+
\1n{1\over2}u^\top R(t)u.\ea\ee
To understand each term in the above, let us denote
$$\ba{ll}
\ss\ds A(s,t)=(A_1(s,t),A_2(s,t),\cds,A_n(s,t)),~B(s,t)=(B_1(s,t),B_2(s,t),
\cds,B_m(s,t)),\\
\ss\ds C(s,t)=(C_1(s,t),C_2(s,t),\cds,C_n(s,t)),~D(s,t)=(D_1(s,t),D_2(s,t),
\cds,D_m(s,t)),\ea\q(s,t)\in\D_*[0,T],$$
where all $A_i(\cd\,,\cd),B_i(\cd\,,\cd),C_i(\cd\,,\cd),D_i(\cd\,,\cd)$ are $\dbR^n$-valued functions. Then, for any $v(\cd\,,\cd)\in C^{1,2}_+(\L)$, we have the following:
$$\ba{ll}
\ss\ds\big[C(\cd\,,t)\bx(t)+D(\cd\,,t)u\big]^\top v_{\bx\bx}(t,\bx)\big[C(\cd\,,t)\bx(t)
+D(\cd\,,t) u\big)\big]\\
\ss\ds=\bx(t)^\top\big[C(\cd\,,t)^\top v_{\bx\bx}(t,\bx)C(\cd\,,t)\big]\bx(t)+\bx(t)^\top\big[C(\cd\,,t)^\top
v_{\bx\bx}(t,\bx)D(\cd\,,t)\big]u\\
\ss\ds\qq+u^\top\big[D(\cd\,,t)^\top v_{\bx\bx}(t,\bx)C(\cd\,,t)\big]\bx(t)+u^\top\big[D(\cd\,,t)^\top
v_{\bx\bx}(t,\bx)D(\cd\,,t)\big]u,\ea$$
with (for all $t\in[0,T]$)
$$\ba{ll}
\ss\ds C(\cd\,,t)^\top v_{\bx\bx}(t,\bx)C(\cd\,,t)\equiv\(C_i(\cd\,,t)^\top v_{\bx\bx}(t,\bx)C_j(\cd\,,t)\)\in\dbS^n,\\
\ss\ds D(\cd\,,t)^\top v_{\bx\bx}(t,\bx)D(\cd\,,t)\equiv\(D_i(\cd\,,t)^\top v_{\bx\bx}(t,\bx)D_j(\cd\,,t)\)\in\dbS^m,\\
\ss\ds\ss\ds D(\cd\,,t)^\top v_{\bx\bx}(t,\bx)C(\cd\,,t)\equiv\(D_i(\cd\,,t)^\top v_{\bx\bx}(t,\bx)C_j(\cd\,,t)\)\in\dbR^{m\times n},\\
\ss\ds C(\cd\,,t)^\top v_{\bx\bx}(t,\bx)D(\cd\,,t)\equiv\(C_i(\cd\,,t)^\top v_{\bx\bx}(t,\bx)D_j(\cd\,,t)\)\in\dbR^{n\times m},\ea$$
and
$$v_\bx(t,\bx)\big[A(\cd\,,t)\bx(t)\1n+\1n B(\cd\,,t)u
\big]
=\big[v_\bx(t,\bx)A(\cd\,,t)\big]\bx(t)+\big[v_\bx(t,x)
B(\cd\,,t)\big]u,$$
with (for all $t\in[0,T]$)
$$\ba{ll}
\ss\ds v_\bx(t,\bx)A(\cd\,,t)=\big(v_x(t,\bx)A_1(\cd\,,t),\cds,v_\bx(t,\bx)
A_n(\cd\,,t)\big)\in\dbR^{1\times n},\\
\ss\ds v_\bx(t,\bx)B(\cd\,,t)=\big(v_x(t,\bx)B_1(\cd\,,t),\cds,v_\bx(t,\bx)
B_m(\cd\,,t)\big)\in\dbR^{1\times m}.\ea$$
From the above, we see that the following condition makes sense:
\bel{positive-condition}
\BR(t,\bx)\equiv D(\cd\,,t)^\top v_{\bx\bx}(t,\bx)D(\cd\,,t)+R(t)\ges\l I_m,\q \hbox{ for some }\l>0,\ee
If such a condition holds, then by denoting
$$\ba{ll}
\ss\ds\BS(t,\bx)=[D(\cd\,,t)^\top v_{\bx\bx}(t,\bx)C(\cd\,,t)]
\bx(t)+[v_\bx(t,\bx)B(\cd\,,t)]^\top,\\
\ss\ds\BQ(t,\bx)=[v_\bx(t,\bx)A(\cd\,,t)]\bx(t)+\bx(t)^\top
[v_\bx(t,\bx)A(\cd\,,t)]^\top\3n+\bx(t)^\top\big[C(\cd\,,t)^\top v_{\bx\bx}(t,\bx)C(\cd\,,t)+Q(t)\big]\bx(t),\ea$$
we have
$$\ba{ll}
\ss\ds2\cH(t,\bx,u,v_\bx(t,\bx),v_{\bx\bx}(t,\bx))= \big[C(\cd\,,t)\bx(t)+D(\cd\,,t)u\big]^\top v_{\bx\bx}(t,\bx)\big[C(\cd\,,t)\bx(t)
+D(\cd\,,t) u\big)\big]\\
\ss\ds\qq\qq\qq\qq\qq\qq\qq+2v_\bx(t,\bx)\big[A(\cd\,,t)
\bx(t)\1n+\1n B(\cd\,,t)u\big]\1n+
\1n\bx(t)^\top Q(t)\bx(t)\1n+
\1n u^\top R(t)u\\
\ss\ds=u^\top\(\big[D(\cd\,,t)^\top v_{\bx\bx}(t,\bx)D(\cd\,,t)\big]+R(t)\)u
\\
\ss\ds\qq+u^\top\(\big[D(\cd\,,t)^\top v_{\bx\bx}(t,\bx)C(\cd\,,t)\big]\bx(t)+[v_\bx(t,\bx)B(\cd\,,t)]^\top
\)\\
\ss\ds\qq+\(\bx(t)^\top\big[C(\cd\,,t)^\top
v_{\bx\bx}(t,\bx)D(\cd\,,t)\big]+[v_\bx(t,\bx)B(\cd\,,t)]\)u\\
\ss\ds\qq+2[v_\bx(t,\bx)A(\cd\,,t)]\bx(t)+\bx(t)^\top\big[C(\cd\,,t)^\top v_{\bx\bx}(t,\bx)C(\cd\,,t)+Q(t)\big]\bx(t)\\
\ss\ds=u^\top\BR(t,\bx)u+u^\top\BS(t,\bx)+\BS(t,\bx)^\top u+\BQ(t,\bx)\\
\ss\ds=\Big|\BR(t,\bx)^{1\over2}\[u+\BR(t,\bx)^{-1}\BS(t,\bx)\]\Big|^2
+\BQ(t,\bx)-\BS(t,\bx)\BR(t,\bx)^{-1}\BS(t,\bx)\\
\ss\ds\ges\BQ(t,\bx)-\BS(t,\bx)\BR(t,\bx)^{-1}\BS(t,\bx)
=\inf_{u\in\dbR^m}2\cH(t,\bx,u,v_\bx(t,\bx),v_{\bx\bx}
(t,\bx))\\
\ss\ds\equiv2\cH(t,\bx,\bar\G(t,\bx),v_\bx(t,\bx),v_{\bx\bx}
(t,\bx)),\ea$$
where
\bel{Causal-strategy-HJB}\ba{ll}
\ss\ds\bar\G(t,\bx)=-\BR(t,\bx)^{-1}\BS(t,\bx)\\
\ss\ds=-\big[D(\cd\,,t)^\top v_{\bx\bx}(t,\bx) D(\cd\,,t)+R(t)\big]^{-1}
\big\{[D(\cd\,,t)^\top v_{\bx\bx}(t,\bx)C(\cd\,,t)]
\bx(t)+[v_\bx(t,\bx)B(\cd\,,t)]^\top\big\}.\ea\ee
Substituting the above into \rf{HJB-PPDE1}, by some straightforward calculations, we get
\bel{HJB-PPDE2}\left\{\2n\ba{ll}
\ss\ds v_t(t,\bx)+{1\over 2}\bx(t)^\top[C(\cd\,,t)^\top v_{\bx\bx}(t,\bx)C(\cd\,,t)]\bx(t)
+{1\over 2}\bx(t)^\top Q(t)\bx(t)\\
\ss\ds\q+[v_\bx(t,\bx)A(\cd\,,t)]\bx(t)-{1\over 2}\(\big[D(\cd\,,t)^\top v_{\bx\bx}(t,\bx)C(\cd\,,t)\bx(t)
+[v_\bx(t,\bx)B(\cd\,,t)]^\top\)^\top\\
\ss\ds\q\times\big[D(\cd,t)^\top v_{\bx\bx}(t,\bx) D(\cd,t)\1n+\1n R(t)\big]^{-1}\(\big[D(\cd,t)^\top v_{\bx\bx}(t,\bx)C(\cd,t)\bx(t)+[v_\bx(t,\bx)
B(\cd\,,t)]^\top\)=0,\\
\ss\ds\qq\qq\qq\qq\qq\qq\qq\qq\qq\qq\qq\qq (t,\bx)\in\sD,\\
\ss\ds v(T,\bx)={1\over 2}\bx(T)^\top G\bx(T).\ea\right.\ee

In the rest of the paper, we shall prove that \rf{HJB-PPDE2} really admits a classical solution by introducing a new type of Riccati equations. With the solution of this Riccati equation, the form of the optimal strategy $\bar\G(\cd\,,\cd)$, defined by \rf{Causal-strategy-HJB}, can be simplified and some interesting phenomenons will be found. More importantly, we will show that the solution of the derived Riccati equation  is exactly the decoupling filed for the optimality system \rf{OP}.

\ms

Recall that $\cS(\dbX)$ is the set of all symmetric bilinear functionals on the Banach space $\dbX\times\dbX$. We let
$$\cS^n_t\deq \cS(\sX_t),\qq\forall t\in[0,T],$$
and let $C([0,T];\cS^n)$ be the set of all $P(\cd)$ having
property that
$$P(t)\in\cS^n_t,\q\forall t\in[0,T],$$
and the map $t\mapsto P(t)$ is continuous with respect to the norm $\|\cd\|_\cS$.
From the definition \rf{value-function} of the value function of Problem (LQ-FSVIE),
we can see that
\bel{def-v-P}v(t,\bx_t(\cd))={P(t)\big(\bx_t(\cd),\bx_t(\cd)
\big)\over 2},\ee
for some bilinear operator-valued function $P(\cd)\in C([0,T];\cS^n)$. We define
\bel{dot-P}\ba{ll}
\ss\ds\dot P(t)(\bx(\cd),\bx(\cd))\deq\lim_{\d\to 0^+}{ P(t+\d)\big([\bx]_{t+\d},[\bx]_{t+\d}\big)
-P(t)\big(\bx(\cd),\bx(\cd)\big)\over \d}=v_t(t,\bx),\q\forall(t,\bx)\in\L,\ea\ee
provided the limit exists. It should be a quadratic form of $\bx_t(\cd)$ again. Moreover, by \rf{gateaux}--\rf{u_oo} and \rf{def-v-P}, we get
\bel{v-x}\ba{ll}
\ss\ds v_\bx(t,\bx_t(\cd))(\eta_t(\cd))=P(t)\big(\bx_t(\cd),\eta_t(\cd)\big),\qq\forall(t,\bx_t(\cd))
\in\L,\,\eta_t(\cd)\in\sX_t,\\
\ss\ds v_{\bx\bx}(t,\bx_t(\cd))(\eta_t(\cd),\eta'_t(\cd))= P(t)(\eta_t(\cd),\eta'_t(\cd)),\qq \forall(t,\bx_t(\cd))\in\L,\,\eta_t(\cd),\eta'_t(\cd)\in\sX_t.\ea\ee
With the representation \rf{def-v-P} and \rf{v-x}, from \rf{HJB-PPDE2}, we identify that $P(\cd)$ should satisfy the following Riccati equation:
\bel{Riccati}\left\{\2n\ba{ll}
\ss\ds\dot P(t)+ P(t)A(\cd\,,t)\d_t+\d_t^\top A(\cd\,,t)^\top P(t)+\d_t^\top C(\cd\,,t)^\top P(t)C(\cd\,,t)\d_t+\d_t^\top Q(t)\d_t
\\
\ss\ds\qq-\big[B(\cd\,,t)^\top P(t)+D(\cd\,,t)^\top P(t)C(\cd\,,t)\d_t\big]^\top\big[R(t)+D(\cd\,,t)^\top P(t)D(\cd\,,t)\big]^{-1}\\
\ss\ds\qq\times\big[B(\cd\,,t)^\top P(t)+D(\cd\,,t)^\top P(t)C(\cd\,,t)\d_t\big]=0,\\
\ss\ds P(T)=G.\ea\right.\ee
The above is understood as an equation for a symmetric bilinear form valued function defined on $\L$. More precisely, for any $(t,\bx_t(\cd))\in\L$, the term $\bx_t(\cd)^\top \dot P(t)\bx_t(\cd)\equiv
\dot P(t)\big(\bx_t(\cd),\bx_t(\cd)\big)$ is understood as in \rf{dot-P}, and other terms are as follows:
$$\ba{ll}
\ss\ds\bx_t(\cd)^\top[P(t)A(\cd\,,t)\d_t]\bx_t(\cd)= P(t)\big(\bx_t(\cd),A(\cd\,,t)\bx_t(t)\big)\\
\ss\ds\qq\qq\qq\qq\qq\q
=P(t)\big(A(\cd\,,t)\bx_t(t),\bx_t(\cd)\big)=\bx_t(\cd)^\top[\d_t^\top A(\cd\,,t)^\top P(t)]\bx_t(\cd),\\
\ss\ds\bx_t(\cd)^\top[\d_t^\top\1n C(\cd\,,t)^\top\1n P(t)C(\cd\,,t)\d_t]\bx_t(\cd)\1n=\1n P(t)\big(C(\cd\,,t)\bx_t(t),
C(\cd\,,t)\bx_t(t)\big)\1n\\
\ss\ds\qq\qq\qq\qq\qq\qq\qq\q=\1n\bx_t(t)^\top\1n\(P(t)\big(C_i(\cd\,,t),
C_j(\cd\,,t)\big)\)
\bx_t(t),\\
\ss\ds\bx_t(\cd)^\top[\d_t^\top Q(t)\d_t]\bx_t(\cd)=\bx_t(t)^\top Q(t)\bx_t(t),\\
\ss\ds D(\cd\,,t)^\top P(t)D(\cd,t)=\(P(t)\big(D_i(\cd\,,t),D_j(\cd\,,t)\big)\)\in\dbS^m,\\
\ss\ds\big[B(\cd\,,t)^\top\1n P(t)\1n+\1n D(\cd\,,t)^\top\1n P(t)C(\cd\,,t)\d_t\big]\bx_t(\cd)=\(P(t)\big(B_j(\cd\,,t),\bx_t(\cd)\big)\1n+\1n P(t)\big(D_j(\cd\,,t),
C(\cd\,,t)\bx_t(t)\big)\)\in\dbR^m.\ea$$

\begin{definition}\label{def-SRS} We call $P(\cd)\in C([0,T];\cS^n)$ a {\it solution} of Riccati equation \rf{Riccati} if it satisfies \rf{Riccati} for any $t\in[0,T]$.
Further, it is called a {\it strongly regular solution},
if, in addition, the following holds:
\bel{SRS}
R(t)+D(\cd\,,t)^\top P(t)D(\cd\,,t)\ges\l I_m;\q t\in[0,T],\q \hbox{ for some }\l>0.\ee
\end{definition}

\ms

We emphasize that for any $t\in[0,T]$, the domain of $P(t)$ is  $\sX_t\times\sX_t$, which is merely a Banach space rather than a Hilbert space, and as $t\in[0,T]$ varies, it changes. Thus, \rf{Riccati} is significantly different from the so-called {\it operator-valued Riccati equation} derived from the LQ control problem for (stochastic) evolution equations (see Li--Yong \cite{Li-Yong1995} and L\"{u} \cite{Lv2019}, for example). To emphasize the new features, we would like to call \rf{Riccati} a {\it path-dependent Riccati equation}.

\begin{theorem}\label{thm:HJB-Riccati} Suppose that  the path-dependent Riccati equation \rf{Riccati} admits a strongly regular solution $P(\cd)\in C([0,T];\cS^n)$.
Then the function $v(\cd\,,\cd)$, defined by
\bel{Causal-strategy}v(t,\bx_t(\cd))={P(t)\big(\bx_t(\cd),\bx_t(\cd)\big)\over 2},\qq\forall(t,\bx_t(\cd))\in\L,\ee
is a classical of the path-dependent HJB equation \rf{HJB-PPDE2}. Moreover, the minimizer $\bar\G(\cd,\cd)$ of the Hamiltonian $\cH$ can be represented as
\bel{Causal-strategy1}\bar\G(t,\bx_t(\cd))=-\big[D(\cd\,,t)^\top P(t)D(\cd\,,t)+R(t)\big]^{-1}\big[D(\cd\,,t)^\top P(t)C(\cd\,,t)\bx_t(t)+B(\cd\,,t)^\top P(t)\bx_t(\cd)\big].\ee

\end{theorem}

\begin{proof}
By \rf{Riccati}, with \rf{dot-P}--\rf{v-x}, we can get that the function $v(\cd,\cd)$ defined by \rf{Causal-strategy}
satisfies \rf{HJB-PPDE2} immediately.
\end{proof}

\begin{remark}
From \autoref{thm:HJB-Riccati}, we have obtained the Riccati equation associated with Problem (LQ-FSVIE). The key of our approach is that $v(t,\bx_t(\cd))$, as well as $v_t(t,\bx_t(\cd))$, $v_\bx(t,\bx_t(\cd))$ and $v_{\bx\bx}(t,\bx_t(\cd))$ are well-defined on the  space $\sX_t$. Then, from the HJB equation, we correctly identify $P(\cd)$.

\end{remark}

\ms

\begin{remark}
If $\alpha(s,r)\equiv\a(r)$ for $\a(\cd\,,\cd)=A(\cd\,,\cd),B(\cd\,,\cd),C(\cd\,,\cd),D(\cd\,,\cd)$,
and $\f(t,s)\equiv x;s\in[t,T]$ for some $x\in\dbR^n$,
then state equation \rf{state} becomes an SDE and Problem (LQ-FSVIE) reduces to a classical stochastic LQ optimal control problem. The corresponding Riccati equation reads (see \cite[Chapter 6]{Yong-Zhou1999}):
\bel{Ri-E}\left\{\2n\ba{ll}
\ss\ds\dot\Si(t)+\Si(t)A(t)+A(t)^\top\Si(t)+C(t)^\top\Si(t)C(t)+Q(t)-[\Si(t)B(t)+C(t)^\top\Si(t)D(t)]\\
\ss\ds\q\times
[R(t)+D(t)^\top\Si(t)D(t)]^{-1}[B(t)^\top\Si(t)+D(t)^\top\Si(t)C(t)]=0, \\
\ss\ds\qq\qq\qq\qq\qq\qq\qq\qq\qq\qq\qq\qq\qq\qq t\in[0,T],\\
\ss\ds\Si(T)=G.\ea\right.\ee
By comparing the above with \rf{Riccati}, we have
$$P(t)(\bx_t(\cd),\bx_t(\cd))=\bx_t(t)^\top\Si(t)\bx_t(t),\qq\forall (t,\bx_t(\cd))\in\L^0.$$
The corresponding value function is given by
$$V(t,\bx_t(\cd))={1\over2}P(t)(\bx_t(\cd),\bx_t(\cd))={1\over 2}\bx_t(t)^\top\Si(t)\bx_t(t),\qq\forall(t,\bx_t(\cd))\in\L^0.$$
Then, all the derivatives of $V(t,\bx_t(\cd))$ are taken for all points in $\L^0$ which is idetified with $[0,T]\times\dbR^n$.. In the above sense, our result recovers the classical one for Problem (LQ-SDE).

\end{remark}

\section{Decoupling the Optimality System}\label{sect-decoupling}

In this section, we shall show that the solution $P(\cd)$ to Riccati equation \rf{Riccati} is exactly the so-called {\it decoupling field} for the optimality system \rf{OP},
which is a coupled FBSVIE. Note that the solution of the path-dependent HJB equation associated with Problem (LQ-FSVIE)
can be represented by \rf{def-v-P}. Thus, the connection between the Hamiltonian system \rf{OP} and the HJB equation \rf{HJB-PPDE2} associated with Problem (LQ-FSVIE) is also established.

\begin{theorem}\label{thm:Decoupling} Let {\rm \ref{ass:H1}--\ref{ass:H3}} hold. Suppose that Riccati equation \rf{Riccati} admits  a strongly regular solution $P(\cd)\in C([0,T];\cS^n)$. Then the solutions of the optimality system \rf{OP} admit the following representation:
\bel{rep}\ba{ll}
\ss\ds\bar X(s)=\bar\cX(s,s),\qq
Y^0(s)=-R(s)\Th(s)\bar \cX(\cd,s),\\
\ss\ds Y(s)=\big[C(\cd\,,s)^\top\1n P(s)C(\cd\,,s)+Q(s)\big]\bar X(s)
+\big[A(\cd\,,s)^\top\1n P(s)+C(\cd\,,s)^\top\1n P(s)D(\cd\,,s)\Th(s)\big]\bar\cX(\cd\,,s),\\
\ss\ds\qq\qq\qq\qq\qq\qq\qq\qq\qq\qq\qq\qq\qq\qq s\in[t,T],\ea\ee
where
\bel{def-Th}\Th(s)\deq-[R(s)+D(\cd\,,s)^\top P(s)D(\cd\,,s)]^{-1}[B(\cd\,,s)^\top P(s)+D(\cd\,,s)^\top P(s)C(\cd\,,s)\d_s],\q s\in[t,T],\ee
with $\bar X(\cd)$ being the unique solution to the closed-loop system
\bel{thm-causal-feedback3}\ba{ll}
\ds\bar X(s)=\bx_t(s)+\int_t^s \big[A(s,\t)\bar X(\t)+B(s,\t)\Th(\t)\bar\cX(\cd\,,\t)\big]d\t\\
\ds\qq\qq+\int_t^s \big[C(s,\t)\bar X(\t)+D(s,\t)\Th(\t)\bar\cX(\cd,\t)\big]dW(\t),\qq s\in[t,T],\ea\ee
and $\bar\cX(\cd\,,\cd)$ being the unique solution to the closed-loop auxiliary system
\bel{thm-causal-feedback2}\ba{ll}
\ss\ds\bar\cX(s,r)=\bx_t(s)+\int_t^r\big[A(s,\t)\bar \cX(\t,\t)+B(s,\t)\Th(\t)\bar\cX(\cd,\t)\big]d\t\\
\ss\ds\qq\qq\qq+\int_t^r \big[C(s,\t)\bar \cX(\t,\t)+D(s,\t)\Th(\t)\bar\cX(\cd,\t)\big]dW(\t),\qq (s,r)\in\D_*[t,T].\ea\ee
Moreover, the optimal control $ \bar u(\cd)$ admits the following causal state feedback representation:
\bel{optimal-causal-strategy}
\bar u(s)=\Th(s)\bar\cX(\cd\,,s),\q s\in[t,T].\ee
\end{theorem}

\begin{remark} By \autoref{thm:Decoupling}, the optimality system \rf{OP} is decoupled, and the optimal control $\bar u(\cd)$ is represented as a causal state feedback (see \rf{optimal-causal-strategy}). Consequently, in principle, the optimal control is practically realizable.
\end{remark}

\begin{remark}\label{remark1} Note that when $B(\cd\,,\cd)\equiv 0$, the optimal control $\bar u(\cd)$ can be represented by
\bel{state-B-0-u}\ba{ll}
\ss\ds\bar u(s)=\Th(s)\bar\cX(\cd\,,s)= -[R(s)+D(\cd\,,s)^\top P(s)D(\cd\,,s)]^{-1}D(\cd\,,s)^\top P(s)C(\cd\,,s)\bar X(s)\\
\ss\ds\qq=:\Th^*(s)\bar X(s),\qq s\in[t,T],\ea\ee
with
\bel{state-B-0}\bar X(s)=\bx_t(s)+\int_t^sA(s,r)\bar X(r)dr+\int_t^s \big[C(s,r)\bar X(r)+D(s,r)\Th^*(r)\bar X(r)\big] dW(r),\q s\in[t,T].\ee
It is particularly worthy of pointing out that though the state equation, \rf{state-B-0} is a non-Markovian system,
the optimal control $\bar u(\cd)$, defined by \rf{state-B-0-u}, can be uniquely determined by the current value of the state.
Thus, when the drift term of \rf{state} dose not contain controls,
the causal feedback representation of the optimal control, obtained by \autoref{thm:Decoupling},
reduces to  a {\it state feedback} (or called a {\it Markovian feedback}), which is interesting.
\end{remark}

\noindent
{\it\textbf{Proof of \autoref{thm:Decoupling}.}}
From the uniqueness of the solution to \rf{thm-causal-feedback3}, we have $\bar X(s)=\bar\cX(s,s),s\in[t,T]$. Let
\bel{eta-zeta}\eta(s)=G\bar\cX(T,T)-\int_s^T\z(r)dW(r),\qq s\in[t,T],\ee
and denote
\begin{align}
&\psi(s)=Q(s)\bar\cX(s,s)+A(T,s)^\top G\bar\cX(T,T)+C(T,s)^{\top}\zeta(s),\nn\\
& \psi^0(s)=B(T,s)^{\top}G\bar\cX(T,T)+D(T,s)^{\top}\zeta(s).\label{def-psi-xi-1}
\end{align}
Then for any $\bx_s\in\sX_s$, by the functional It\^o's formula (see \autoref{Ito-RE}), we have
$$\ba{ll}
\ss\ds d\big\{[\bx_s]_r(\cd)^\top P(r)\bar\cX(\cd\,,r)\big\}\equiv d \big\{P(r)\big([\bx_s]_r(\cd),\bar\cX(\cd\,,r)\big)\big\}\\
\ss\ds\q=\big\{[\bx_s]_r(\cd)^\top\dot P(r)\bar\cX(\cd\,,r)+[\bx_s]_r(\cd)^\top P(r) [A(\cd\,,r)\bar\cX(r,r)+B(\cd\,,r)\Th(r)\bar\cX(\cd\,,r)]\big\}dr\\
\ss\ds\qq+[\bx_s]_r(\cd)^\top P(r)\big[ C(\cd\,,r)\bar\cX(r,r)+D(\cd\,,r)\Th(r)\bar\cX(\cd\,,r)\big] dW(r)\\
\ss\ds\q=\big\{P(r)\big([\bx_s]_r(\cd), A(\cd\,,r)\bar\cX(r,r)+B(\cd\,,r)\Th(r)\bar\cX(\cd\,,r)\big)
-\bx_s(r)^\top C(\cd,r)^\top P(r)C(\cd,r)\bar\cX(r,r)\\
\ss\ds\q\qq-\bx_s(r)^\top Q\bar\cX(r,r)-\bar\cX(\cd\,,r)^\top P(r)A(\cd\,,r)\bx_s(r)-[\bx_s]_r(\cd)^\top P(r)A(\cd,r)\bar\cX(r,r)\\
\ss\ds\q\qq-\bx_s(r)^\top C(\cd\,,r)^\top P(r) D(\cd\,,r)\Th(r)\bar\cX(\cd,r)-\bx_s(r)^\top P(r) B(\cd\,,r)\Th(r)\bar\cX(\cd,r)\big\}dr\\
\ss\ds\q\qq+[\bx_s]_r(\cd)^\top P(r)\big[ C(\cd\,,r)\bar\cX(r,r)+D(\cd\,,r)\Th(r)\bar\cX(\cd\,,r)]dW(r)\\
\ss\ds\q=-\big\{\bx_s(r)^\top A(\cd\,,r)^\top P(r)\bar\cX(\cd\,,r)+\bx_s(r)^\top C(\cd\,,r)^\top P(r)C(\cd\,,r)\bar\cX(r,r)\\
\ss\ds\q\qq+\bx_s(r)^\top Q(r)\bar\cX(r,r)+\bx_s(r)^\top C(\cd\,,r)^\top P(r)D(\cd\,,r)\Th(r)\bar\cX(\cd\,,r)\big\}dr\\
\ss\ds\q\qq+[\bx_s]_r(\cd)^\top P(r)[C(\cd\,,r)\bar\cX(r,r)+D(\cd\,,r)\Th(r)\bar\cX(\cd\,,r)] dW(r),\q r\in[s,T].\ea$$
Combining the above with \rf{eta-zeta}, we have
\bel{thm:Decoupling-proof1}\ba{ll}
\ds[\bx_s]_r(\cd)^\top P(r)\bar\cX(\cd\,,r)-\bx_s(T)^\top \dbE_r[G\bar\cX(T,T)]\\
\ds\q=\int_r^T\Big\{\bx_s(\t)^\top A(\cd\,,\t)^\top P(\t)\bar\cX(\cd\,,\t)+\bx_s(\t)^\top C(\cd\,,\t)^\top P(\t)C(\cd\,,\t)\bar\cX(\t,\t)\\
\ds\qq\qq+\bx_s(\t)^\top Q(\t)\bar\cX(\t,\t)+\bx_s(\t)^\top C(\cd\,,\t)^\top P(\t) D(\cd\,,\t)\Th(\t)\bar\cX(\cd\,,\t)\Big\}d\t\\
\ds\qq-\int_r^T\Big\{[\bx_s]_\t(\cd)^\top P(\t)[ C(\cd\,,\t)\bar\cX(\t,\t)+D(\cd\,,\t)\Th(\t)\bar\cX(\cd\,,\t)]
-\bx_s(T)^\top\z(\t)\Big\}dW(\t).\ea\ee
Denote
\bel{thm:Decoupling-proof11}\ba{ll}
\ss\ds Y^A(s)=A(\cd\,,s)^\top P(s)\bar\cX(\cd\,,s)-A(T,s)^\top \dbE_s[G\bar\cX(T,T)],\\
\ss\ds Z^A(s,r)=[A(\cd\,,s)]_rP(r)\big[C(\cd\,,r)\bar\cX(r,r)+D(\cd\,,r)
\Th(r)\bar\cX(\cd\,,r)\big]-A(T,s)^\top\z(r),\\
\ss\ds Y^B(s)=B(\cd\,,r)^\top P(r)\bar\cX(\cd\,,r)-B(T,s)^\top \dbE_s[G\bar\cX(T,T)],\\
\ss\ds Z^B(s,r)=[B(\cd\,,s)]_r^\top P(r)\big[ C(\cd\,,r)\bar\cX(r,r)+D(\cd\,,r)\Th(r)\bar\cX(\cd\,,r)\big]
-B(T,s)^\top\z(r),\\
\ss\ds Y^C(s)=C(\cd\,,r)^\top P(r)\bar\cX(\cd\,,r)-C(T,s)^\top \dbE_s[G\bar\cX(T,T)],\\
\ss\ds Z^C(s,r)=[C(\cd\,,s)]_r^\top P(r)\big[C(\cd\,,r)\bar\cX(r,r)+D(\cd\,,r)\Th(r)\bar\cX(\cd\,,r)]
-C(T,s)^\top\z(r),\\
\ss\ds Y^D(s)=D(\cd\,,r)^\top P(r)\bar\cX(\cd\,,r)-D(T,s)^\top \dbE_s[G\bar\cX(T,T)],\\
\ss\ds Z^D(s,r)=[D(\cd\,,s)]_r^\top P(r)\big[C(\cd\,,r)\bar\cX(r,r)+D(\cd\,,r)\Th(r)\bar\cX(\cd\,,r)]
-D(T,s)^\top\z(r).\ea\ee
Taking $\bx_s(\cd)=A(\cd\,,s)$ and $r=s$ in \rf{thm:Decoupling-proof1}, we have
\bel{thm:Decoupling-proof22-1}\ba{ll}
\ds Y^A(s)=\int_s^T\Big\{A(r,s)^\top Y^A(r)+A(r,s)^\top\dbE_r[ A(T,r)^\top G\bar\cX(T,T)]+A(r,s)^\top Z^C(r,r)\\
\ds\qq+A(r,s)^\top C(T,r)^\top\zeta(r)+A(r,s)^\top Q(r)\bar\cX(r,r)\Big\}dr-\int_s^T Z^A(s,r)dW(r)\\
\ds=\int_s^T\big\{A(r,s)^\top\dbE_r[\psi(r)]+A(r,s)^\top Y^A(r)+A(r,s)^\top Z^C(r,r)\big\}dr-\int_s^TZ^A(s,r)dW(r),\ea\ee
where $\psi(\cd)$ is defined by \rf{def-psi-xi-1}.
Similarly, by taking $\bx_s(\cd)=B(\cd\,,s),C(\cd\,,s),D(\cd\,,s)$ in \rf{thm:Decoupling-proof1}, we have
\bel{thm:Decoupling-proof22}\ba{ll}
\ss\ds Y^B(s)=\2n\int_s^T\3n\big\{B(r,s)^\top\dbE_r[\psi(r)]\1n+\1n B(r,s)^\top Y^A(r)\1n+\1n B(r,s)^\top Z^C(r,r)\big\}dr\1n-\2n\int_s^T\3n Z^B(s,r)dW(r),\\
\ss\ds Y^C(s)=\2n\int_s^T\3n\big\{C(r,s)^\top\dbE_r[\psi(r)]\1n+\1n C(r,s)^\top Y^A(r)\1n+\1n C(r,s)^\top Z^C(r,r)\big\}dr\1n-\2n\int_s^T\3n Z^C(s,r)dW(r),\\
\ss\ds Y^D(s)=\2n\int_s^T\3n\big\{D(r,s)^\top\dbE_r[\psi(r)]\1n+\1n D(r,s)^\top Y^A(r)\1n+\1n D(r,s)^\top Z^C(r,r)\big\}dr\1n-\2n\int_s^T\3n Z^D(s,r)dW(r).\ea\ee
Let $(Y(\cd),Z(\cd\,,\cd),Y^0(\cd),Z^0(\cd\,,\cd))$ be the unique adapted solution to the following BSVIE:
\bel{optimality-system-coupled-RW-de11}\left\{\2n\ba{ll}
\ds Y(s)=\psi(s)+\int_s^T\big[A(r,s)^\top Y(r)+C(r,s)^\top Z(r,s)\big]dr-\int_s^T Z(s,r)dW(r),\\
\ds Y^0(s)=\psi^0(s)+\int_s^T\big[B(r,s)^\top Y(r) +D(r,s)^\top Z(r,s)\big]dr-\int_s^TZ^0(s,r)dW(r).\ea\right.\ee
Then by \autoref{Prop:Type-I-II}, we have
$$Y(s)=\dbE_s[\psi(s)]+Y^A(s)+Z^C(s,s),\q Y^0(s)=\dbE_s[\psi^0(s)]+Y^B(s)+Z^D(s,s),\q s\in[t,T].$$
Substituting \rf{thm:Decoupling-proof11} into the above, we get
$$\ba{ll}
\ss\ds Y(s)=Q(s)\bar\cX(s,s)+A(T,s)^\top G\dbE_s[\bar\cX(T,T)]+C(T,s)^\top\z(s)+A(\cd\,,s)^\top P(s)\bar\cX(\cd\,,s)\\
\ss\ds\qq\qq-A(T,s)^\top\dbE_s[G\bar\cX(T,T)]+C(\cd\,,s)^\top P(s)[ C(\cd\,,s)\bar\cX(s,s)+D(\cd\,,s)\Th(s)\bar\cX(\cd\,,s)]
-C(T,s)^\top\z(s)\\
\ss\ds\qq=A(\cd\,,s)^\top P(s)\bar\cX(\cd\,,s)+C(\cd\,,s)^\top P(s) C(\cd\,,s)\bar\cX(s,s)+Q(s)\bar\cX(s,s)\\
\ss\ds\qq\qq+C(\cd\,,s)^\top P(s)D(\cd\,,s)\Th(s)\bar\cX(\cd\,,s),\ea$$
and
$$\ba{ll}
\ss\ds Y^0(s)=B(T,s)^{\top}G\dbE_s[\bar\cX(T,T)]+D(T,s)^{\top}\z(s)
+B(\cd\,,s)^\top P(s)\bar\cX(\cd\,,s)-B(T,s)^\top \dbE_s[G\bar\cX(T,T)]\\
\ss\ds\qq\qq\q+D(\cd\,,s)^\top P(s)[ C(\cd\,,s)\bar\cX(s,s)+D(\cd\,,s)\Th(s)\bar\cX(\cd\,,s)]
-D(T,s)^\top\z(s)\\
\ss\ds\qq\q=B(\cd\,,s)^\top P(s)\bar\cX(\cd\,,s)+D(\cd\,,s)^\top P(s)[C(\cd\,,s)\bar\cX(s,s)+D(\cd\,,s)\Th(s)\bar\cX(\cd\,,s)]\\
\ss\ds\qq\q=-R(s)\Th(s)\bar\cX(\cd\,,s).\ea$$
Then the desired results can be obtained easily. $\hfill\qed$

\ms

Let $\sX^*([s,T];\dbR^n)$ be the dual space of $\sX([s,T];\dbR^n)$;
that is the space consisting of all the bounded linear functionals on $\sX([s,T];\dbR^n)$.
Clearly, for \as\, $\om\in\Om$,
\bel{}
\dbE_s[Y(\cd)]|_{[s,T]}(\om)\in \sX^*([s,T];\dbR^n) \q\hbox{and}\q Z(\cd,s)|_{[s,T]}(\om)\in \sX^*([s,T];\dbR^n),
\ee
by letting
\begin{align}
\lan \dbE_s[Y(\cd)]|_{[s,T]}, \bx_s(\cd)\ran(\om)&\deq \int_s^T\bx_s(r)^\top \dbE_s[ Y(r)](\om)dr=\dbE_s\int_s^T\bx_s(r)^\top Y(r)dr(\om),\nn\\
\lan  Z(\cd,s)|_{[s,T]}, \bx_s(\cd)\ran(\om)&\deq\int_s^T\bx_s(r)^\top Z(r,s)dr(\om),\q \forall\bx_s\in\sX([s,T];\dbR^n).
\end{align}
The following result provides a representation for $(\dbE_s[Y(\cd)]|_{[s,T]},Z(\cd,s)|_{[s,T]})$ in the space $\sX^*([s,T];\dbR^n)$.

\begin{theorem}\label{corollary-YZ}
For any $s\in[t,T]$, the following equalities hold
\begin{align}
\dbE_s[Y(\cd)]|_{[s,T]}&= P(s)\bar\cX(\cd,s)- \dbE_s[G\bar\cX(T,T)],\nn\\
Z(\cd,s)|_{[s,T]}&= P(s)C(\cd,s)\bar\cX(s,s)+P(s)D(\cd,s)\Th(s)\bar\cX(\cd,s)- \zeta(s),
\label{corollary-YZ-main1}
\end{align}
in the space $\sX^*([s,T];\dbR^n)$; that is for any $\bx_s\in C([s,T];\dbR^n)$ and \as\, $\om\in\Om$,
\begin{align}
\dbE_s\int_s^T\bx_s(r)^\top Y(r)dr(\om)&=\big\{\bx_s(\cd)^\top P(s)\bar\cX(\cd,s)-\bx_s(T)^\top \dbE_s[G\bar\cX(T,T)]\big\}(\om),\nn\\
\int_s^T\bx_s(r)^\top Z(r,s)dr(\om)&=\big\{\bx_s(\cd)^\top P(s)C(\cd,s)\bar\cX(s,s)+\bx_s(\cd)^\top P(s)D(\cd,s)\Th(s)\bar\cX(\cd,s)\nn\\
&\q-\bx_s(T)^\top \zeta(s)\big\}(\om).
\label{corollary-YZ-main2}
\end{align}
\end{theorem}

\begin{proof} By the same arguments as in the proof of \autoref{Prop:Type-I-II}, we have
\begin{align}
\dbE_s\int_s^T\bx_s(r)^\top Y(r)dr&=\int_s^T\bx_s(r)^\top\Big\{A(T,r)^\top\dbE_r[G\bar\cX(T,T)]+ C(T,r)^\top\zeta(r)+ Q(r)\bar\cX(r,r)\nn\\
&\q+\dbE_r\int_r^TA(\t,r)^\top Y(\t)d\t+\int_r^TC(\t,r)^\top Z(\t,r)d\t\Big\}dr\nn\\
&\q-\int_s^T\int_r^T\bx_s(\t)^\top Z(\t,r)d\t dW(r),\label{thm:Decoupling-proof2}
\end{align}
and
\begin{align}
& Y^A(s)=\dbE_s\int_s^T A(r,s)^\top Y(r)dr,\q  Z^C(s,r)=\int_r^T C(\t,s)^\top Z(\t,r)d\t,
\end{align}
where $(Y^A(\cd), Z^C(\cd,\cd))$ is uniquely determined by BSVIEs \rf{thm:Decoupling-proof22-1}--\rf{thm:Decoupling-proof22}.
Thus, we can rewrite \rf{thm:Decoupling-proof2} as
\begin{align}
&\dbE_s\int_s^T\bx_s(r)^\top Y(r)dr=\int_s^T\bx_s(r)^\top\Big\{A(T,r)^\top\dbE_r[G\bar\cX(T,T)]+ C(T,r)^\top\zeta(r)+ Q(r)\bar\cX(r,r)\nn\\
&\qq+Y^A(r)+Z^C(r,r)\Big\}dr-\int_s^T\int_r^T\bx(\t)^\top Z(\t,r)d\t dW(r).\label{thm:Decoupling-proof23}
\end{align}
On the other hand, by  \rf{thm:Decoupling-proof1}, we have
\begin{align}
&\bx_s(\cd)^\top P(s)\bar\cX(\cd,s)-\bx_s(T)^\top \dbE_s[G\bar\cX(T,T)]\nn\\
&\q=\int_s^T\Big\{\bx_s(r)^\top A(\cd,r)^\top P(r)\bar\cX(\cd,r)+\bx_s(r)^\top C(\cd,r)^\top P(r)C(\cd,r)\bar\cX(r,r)\nn\\
&\qq\qq+\bx_s(r)^\top Q(r)\bar\cX(r,r)+\bx_s(r)^\top C(\cd,r)^\top P(r) D(\cd,r)\Th(r)\bar\cX(\cd,r)\Big\}dr\nn\\
&\qq-\int_s^T\Big\{[\bx_s(\cd)]_r^\top P(r)[ C(\cd,r)\bar\cX(r,r)+D(\cd,r)\Th(r)\bar\cX(\cd,r)]-\bx_s(T)^\top \zeta(r)\Big\} dW(r).
\end{align}
Substituting  \rf{thm:Decoupling-proof11} into the above, we have
\begin{align}
&\bx_s(\cd)^\top P(s)\bar\cX(\cd,s)-\bx_s(T)^\top \dbE_s[G\bar\cX(T,T)]\nn\\
&\q=\int_s^T\bx_s(r)^\top\Big\{A(T,r)^\top\dbE_r[G\bar\cX(T,T)]+ C(T,r)^\top\zeta(r)+ Q(r)\bar\cX(r,r)+Y^A(r)+Z^C(r,r)\Big\}dr\nn\\
&\qq-\int_s^T\Big\{[\bx_s(\cd)]_r^\top P(r)[ C(\cd,r)\bar\cX(r,r)+D(\cd,r)\Th(r)\bar\cX(\cd,r)]-\bx_s(T)^\top \zeta(r)\Big\} dW(r).
\label{thm:Decoupling-proof24}
\end{align}
Comparing \rf{thm:Decoupling-proof23} and \rf{thm:Decoupling-proof24},
we get \rf{corollary-YZ-main1} immediately.
%\begin{align}
%\dbE_s\int_s^T\bx(r)^\top Y(r)dr&=\lan P(s), (\bar\cX_s,\bx_s)\ran-\bx(T)^\top \dbE_s[G\bar\cX(T,T)],\nn\\
%
%\int_s^T\bx(r)^\top Z(r,s)dr&=\lan P(s), (\bx_s,C_s)\ran\bar\cX(s,s)+\lan P(s), (\bx_s,D_s)\ran\Th(s)\bar\cX(s,s)-\bx(T)^\top \zeta(s).
%\end{align}
\end{proof}

\begin{remark}
If we let $G=0$, then
\begin{align}
\dbE_s[Y(\cd)]|_{[s,T]}= P(s)\bar\cX(\cd,s),\q
Z(\cd,s)|_{[s,T]}= P(s)C(\cd,s)\bar\cX(s,s)+P(s)D(\cd,s)\Th(s)\bar\cX(\cd,s),
\end{align}
which is very similar to the results of classical stochastic LQ control problems (see \cite[Chapter 6]{Yong-Zhou1999}, for example).
Thus,  it should be more natural to view the solution $(Y(\cd),Z(\cd,\cd))$ as an element  in the space $C^*([0,T];\dbR^n)$.
\end{remark}

\section{Well-posedness of the path-dependent Riccati equation}\label{sect-wellposedness}

In this section, we shall establish the well-posedness of path-dependent Riccati equation \rf{Riccati} which is rewritten here, for convenience:
\bel{Riccati-Equation-RW}\left\{\ba{ll}
\ss\ds\dot P(t)+ P(t)A(\cd\,,t)\d_t+\d_t^\top A(\cd\,,t)^\top P(t)+\d_t^\top C(\cd\,,t)^\top P(t)C(\cd\,,t)\d_t+\d_t^\top Q(t)\d_t
\\
\ss\ds\qq-\big[D(\cd\,,t)^\top P(t)C(\cd\,,t)\d_t+B(\cd\,,t)^\top P(t)\big]^\top\big[D(\cd\,,t)^\top P(t)D(\cd\,,t)+R(t)\big]^{-1}\\
\ss\ds\qq\times\big[D(\cd\,,t)^\top P(t)C(\cd\,,t)\d_t+B(\cd\,,t)^\top P(t)\big]=0,\\
\ss\ds P(T)=G.\ea\right.\ee
We assume the following condition.

\begin{taggedassumption}{(H4)}\label{ass:H4}\rm
The weighting matrices $Q(\cd)$, $R(\cd)$ and $G$ satisfy
\bel{standard-condition}
Q(s)\ges 0,\q R(s)\ges\a I_m,\q s\in\dbT;\q G\ges 0,
\ee
where $\a>0$ is a given constant.
\end{taggedassumption}

\begin{theorem}\label{theorem:well-posedness-RE} Let {\rm \ref{ass:H1}--\ref{ass:H2}} and {\rm \ref{ass:H4}} hold.
Then the path-dependent Riccati equation \rf{Riccati-Equation-RW} admits a unique strongly regular solution $P(\cd)\in C([0,T];\cS^n)$. Moreover,
$$P(t)(\bx_t(\cd),\bx_t(\cd))\ges 0,\qq\forall(t,\bx_t)\in\L.$$

\end{theorem}

\begin{remark} For the corresponding results of  \autoref{theorem:well-posedness-RE} in the SDE setting, we refer the reader to \cite[Chapter 6]{Yong-Zhou1999}, in which \ref{ass:H4} was called a {\it standard condition}. It is known that \ref{ass:H4} implies that the uniformly convexity condition \ref{ass:H3} holds. An interesting question is whether Riccati equation \rf{Riccati-Equation-RW} has the well-posedness under \ref{ass:H3},
as did by Sun--Li--Yong \cite{Sun-Li-Yong2016} in the SDE setting. We shall explore that in the near future.
\end{remark}

To establish the well-posedness of \rf{Riccati-Equation-RW}, we introduce the following {\it path-dependent Lyapunov equation}:
\bel{lem:Solvability-RE-main1}\left\{\2n\ba{ll}
\ss\ds\dot P(t)+P(t)A(\cd\,,t)\d_t+\d_t^\top A(\cd\,,t)^\top P(t)+\d_t^\top C(\cd\,,t)^\top P(t)C(\cd\,,t)\d_t+\d_t^\top Q(t)\d_t\\
\ss\ds\q-\big[P(t)B(\cd,t)+\d_t^\top C(\cd,t)^\top P(t)D(\cd,t)\big]\Psi(t)-\Psi(t)^\top\big[B(\cd,t)^\top P(t)
+D(\cd,t)^\top P(t)C(\cd,t)\d_t\big]\\
\ss\ds\q+\Psi(t)^\top\big[R(t)+D(\cd\,,t)^\top P(t)D(\cd\,,t)\big]\Psi(t)=0,\\
\ss\ds P(T)=G,\ea\right.\ee
where $\Psi(t):\sX_t\to\dbR^m$ is a bounded linear functional for any $t\in[0,T]$. We call $P(\cd)\in C([0,T];\cS^n)$ a solution of \rf{lem:Solvability-RE-main1} if for any $\bx,\bx'\in\sX_0$, with $[\bx]_t=\bx(t){\bf1}_{[0,t)}+\bx{\bf1}_{[t,T]}$ and $[\bx']_t=\bx'(t){\bf1}_{[0,t)}+\bx'{\bf1}_{[t,T]}$,
\bel{Lyapunov}\left\{\2n\ba{ll}
\ss\ds[\bx]_t(\cd)^\top\1n\dot P(t)[\bx']_t(\cd)\1n+\1n[\bx]_t(\cd)^\top\1n P(t)A(\cd\,,t)\bx'(t)\1n+\1n\bx(t)^\top\1n A(\cd\,,t)^\top P(t)[\bx']_t(\cd)\1n\\
\ss\ds\q+\bx(t)^\top\1n C(\cd\,,t)^\top\1n P(t)C(\cd\,,t)\bx'(t)+\1n\bx(t)^\top\1n Q(t)\bx'(t)\\
\ss\ds\q-[\bx]_t(\cd)^\top P(t) B(\cd\,,t)\Psi(t)[\bx']_t(\cd)
-\bx(t)^\top C(\cd\,,t)^\top P(t) D(\cd\,,t)\Psi(t)[\bx']_t(\cd)\\
\ss\ds\q-\{\Psi(t)[\bx]_t(\cd)\}^\top B(\cd\,,t)^\top P(t)\bx'(t)-\{\Psi(t)[\bx]_t(\cd)\}^\top D(\cd\,,t)^\top P(t)C(\cd\,,t)\bx'(t)\\
\ss\ds\q+\{\Psi(t)[\bx]_t(\cd)\}^\top\1n R(t)\{\Psi(t)[\bx']_t(\cd)\}\1n+\{\Psi(t)[\bx]_t(\cd)\}^\top\1n D(\cd\,,t)^\top\1n P(t)D(\cd\,,t)\{\Psi(t)[\bx']_t(\cd)\}\1n=\1n0,\\
\ss\ds P(T)([\bx]_T,[\bx']_T)=\bx(T)^\top G\bx'(T).\ea\right.\ee

\begin{lemma}\label{lem:Solvability-RE} Let {\rm\ref{ass:H1}--\ref{ass:H2}} hold. Then the path-dependent  Lyapunov equation \rf{lem:Solvability-RE-main1} admits a unique solution $P(\cd)\in C([0,T];\cS^n)$. Moreover, if $ G\ges0$, $R(\cd)\ges0$ and $ Q(\cd)\ges 0$, we have
\bel{lem:Solvability-RE-main2}P(t)\big(\bx_t(\cd),\bx_t(\cd)\big) \ges0,\qq\forall\bx_t\in\sX_t.\ee
\end{lemma}

\begin{proof} We first prove the existence and the uniqueness of the solution to Lyapunov equation \rf{lem:Solvability-RE-main1}.
For any $\bx,\bx'\in\sX_0$ with $\|\bx(\cd)\|,\|\bx'(\cd)\|=1$, from \rf{Lyapunov} we have
$$\ba{ll}
\ss\ds[\bx]_t(\cd)^\top P(t)[\bx']_t(\cd)=\bx(T)^\top G\bx'(T)+\int_t^T\Big\{[\bx]_s(\cd)^\top P(s)A(\cd,s)\bx'(s)\\
\ss\ds\q+\bx(s)^\top A(\cd\,,s)^\top P(s)[\bx']_s(\cd)+\bx(s)^\top C(\cd\,,s)^\top P(s)C(\cd\,,s)\bx'(s)+\bx(s)^\top Q(s)\bx'(s)\\
\ss\ds\q-[\bx]_s^\top P(s)B(\cd\,,s)\Psi(s)[\bx']_s(\cd)
-\bx(s)^\top C(\cd\,,s)^\top P(s) D(\cd\,,s)\Psi(s)[\bx']_s(\cd)\\
\ss\ds\q-\{\Psi(s)[\bx]_s(\cd)\}^\top B(\cd\,,s)^\top P(s)\bx'(s)-\{\Psi(s)[\bx]_s(\cd)\}^\top D(\cd\,,s)^\top P(s)C(\cd\,,s)\bx'(s)\\
\ss\ds\q+\{\Psi(s)[\bx]_s(\cd)\}^\top R(s)\{\Psi(s)[\bx']_s(\cd)\}+\{\Psi(s)[\bx]_s(\cd)\}^\top D(\cd\,,s)^\top P(s)D(\cd,s)\{\Psi(s)[\bx']_s(
\cd)\}\Big\}ds.\ea$$
For any $\bar P(\cd)\in C([0,T];\cS^n)$, denote the operator-valued function $P(\cd)$ by
\bel{lem:Solvability-RE-proof1}\ba{ll}
\ss\ds[\bx]_t(\cd)^\top P(t)[\bx']_t(\cd)=\bx(T)^\top G\bx'(T)+\int_t^T\Big\{[\bx]_s(\cd)^\top\bar P(s)A(\cd,s)\bx'(s)\\
\ss\ds\q+\bx(s)^\top A(\cd\,,s)^\top \bar P(s)[\bx']_s(\cd)+\bx(s)^\top C(\cd\,,s)^\top \bar P(s)C(\cd\,,s)\bx'(s)+\bx(s)^\top Q(s)\bx'(s)\\
\ss\ds\q-[\bx]_s^\top\bar P(s)B(\cd\,,s)\Psi(s)[\bx']_s(\cd)
-\bx(s)^\top C(\cd\,,s)^\top\bar P(s) D(\cd\,,s)\Psi(s)[\bx']_s(\cd)\\
\ss\ds\q-\{\Psi(s)[\bx]_s(\cd)\}^\top B(\cd\,,s)^\top \bar P(s)\bx'(s)-\{\Psi(s)[\bx]_s(\cd)\}^\top D(\cd\,,s)^\top \bar P(s)C(\cd\,,s)\bx'(s)\\
\ss\ds\q+\{\Psi(s)[\bx]_s(\cd)\}^\top R(s)\{\Psi(s)[\bx']_s(\cd)\}+\{\Psi(s)[\bx]_s(\cd)\}^\top D(\cd\,,s)^\top\bar P(s)D(\cd,s)\{\Psi(s)[\bx']_s(
\cd)\}\Big\}ds.\ea\ee
From the fact that $\bar P(\cd)$ is symmetric, we see that $P(\cd)\in C([0,T];\cS^n)$. Thus, the map $\G:C([0,T];\cS^n)\to C([0,T];\cS^n)$
$$P(\cd)=\G\big(\bar P(\cd)\big)$$
is well-defined. For any $\bar P_1(\cd)$, $\bar P_2(\cd)\in C([0,T];\cS^n)$, denote
$$P_i(\cd)=\G(\bar P_i(\cd)),\q\D P(\cd)=P_1(\cd)-P_2(\cd),\q\D\bar P(\cd)=\bar P_1(\cd)-\bar P_2(\cd).$$
Then from \rf{lem:Solvability-RE-proof1}, taking $\bx^\prime=\bx$, we have
$$\ba{ll}
\ss\ds[\bx]_t(\cd)^\top\D P(t)[\bx]_t(\cd)=\int_t^T
\Big\{2[\bx]_s(\cd)^\top\D\bar P(s)A(\cd\,,s)\bx(s) -2[\bx]_s(\cd)^\top\D\bar P(s) B(\cd\,,s)\Psi(s)[\bx]_s(\cd)\\
\ss\ds\qq\qq\qq\qq+\bx(s)^\top C(\cd\,,s)^\top\D\bar P(s)C(\cd\,,s)\bx(s)-2\bx(s)^\top C(\cd\,,s)^\top\D\bar P(s) D(\cd\,,s)\Psi(s)[\bx]_s(\cd)\\
\ss\ds\qq\qq\qq\qq+\{\Psi(s)[\bx]_s(\cd)\}^\top D(\cd\,,s)^\top\D\bar P(s)D(\cd\,,s)\{\Psi(s)[\bx]_s(\cd)\}\Big\}ds.\ea$$
Thus, by \autoref{lem:norm}, we get
$$\sup_{s\in[t,T]}\|\D P(s)\|\les K|T-t|\sup_{s\in[t,T]}\|\D \bar P(s)\|,$$
where $K>0$, only depending on the norms of the coefficients, is a fixed constant. The by the contraction mapping theorem, the existence and the uniqueness of solutions to Lyapunov equation \rf{lem:Solvability-RE-main1} can be obtained.

\ms

We now prove \rf{lem:Solvability-RE-main2}.
Let $\cX(\cd\,,\cd)$ be the unique solution to the following SVIE:
\begin{align}
\cX(s,r)&=\bx_t(s)+\int_t^r [ A(s,\t)\cX(\t,\t)- B(s,\t)\Psi(\t)\cX(\cd,\t)] d\t\nn\\
&\q+\int_t^r  [C(s,\t)\cX(\t,\t)- D(s,\t)\Psi(\t)\cX(\cd,\t)] dW(\t),\q (s,r)\in\D_*[t,T].
\end{align}
Then by the functional It\^{o}'s formula (see \autoref{Ito-RE}), we have
\begin{align}
 P(s)\big(\cX(\cd,s),\cX(\cd,s)\big)&=\dbE_s\Big\{\int_s^T\big [\lan  Q(r)\cX(r,r),\cX(r,r)\ran+\lan  R(r)\Psi(r)\cX(\cd,r),\Psi(r)\cX(\cd,r)\ran\big] dr\nn\\
&\q +\lan G\cX(T,T),\cX(T,T)\ran\Big\}.
\end{align}
In particular, taking $s=t$, we have (noting $\cX(\cd,t)=\bx_t(\cd)$)
\begin{align}
P(t) \big(\bx_t(\cd),\bx_t(\cd)\big)&=\dbE_t\Big\{\int_t^T\big [\lan  Q(r)\cX(r,r),\cX(r,r)\ran+\lan  R(r)\Psi(r)\cX(\cd,r),\Psi(r)\cX(\cd,r)\ran\big] dr\nn\\
&\q +\lan G\cX(T,T),\cX(T,T)\ran\Big\}\ges 0,
\end{align}
where the last inequality is due to the facts $ G\ges0$, $ Q(\cd)\ges 0$ and $R(\cd)\ges 0$.
\end{proof}

\ms\noindent
{\bf Proof of \autoref{theorem:well-posedness-RE}.}
The uniqueness of the solution to Riccati equation \rf{Riccati-Equation-RW} can be obtained by a standard method.
We now prove the existence of a solution to Riccati equation \rf{Riccati-Equation-RW} by a iterative method.
Denote
\bel{}
\Psi(t)=[D(\cd\,,t)^\top P(t)D(\cd\,,t)+R(t)]^{-1}[D(\cd\,,t)^\top P(t)C(\cd\,,t)\d_t+B(\cd\,,t)^\top P(t)].
\ee
Then from the fact
\bel{}
[D(\cd\,,t)^\top P(t)D(\cd\,,t)+R(t)]^{-1}\Psi(t)=
[D(\cd\,,t)^\top P(t)C(\cd\,,t)\d_t+B(\cd\,,t)^\top P(t)],
\ee
it is easily checked that
\rf{Riccati-Equation-RW} is equivalent to the following
\bel{theorem:well-posedness-RE-proof1}\left\{\begin{aligned}
&\dot{P}(t)+P(t)A(\cd\,,t)\d_t+\d_t^\top A(\cd\,,t)^\top P(t)-P(t)B(\cd,t)\Psi(t)-\Psi(t)^\top B(\cd,t)^\top P(t)\nn\\
&\q+\d_t^\top C(\cd\,,t)^\top P(t)C(\cd\,,t)\d_t-\d_t^\top C(\cd,t)^\top P(t) D(\cd,t)\Psi(t)-\Psi(t)^\top D(\cd,t)^\top P(t)C(\cd,t)\d_t\\
&\q+\Psi(t)^\top D(\cd\,,t)^\top P(t)D(\cd\,,t)\Psi(t)+\d_t^\top Q(t)\d_t+\Psi(t)^\top R(t)\Psi(t)=0,\\
&P(T)= G.
\end{aligned}\right.\ee
By \autoref{lem:Solvability-RE}, with $\Psi(\cd)\equiv 0$, the following equation admits a unique solution
$P_0(\cd)$:
\bel{}\left\{\begin{aligned}
&\dot{P}_0(t)+P_0(t)A(\cd,t)\d_t+\d_t^\top A(\cd,t)^\top P_0(t)+\d_t^\top C(\cd,t)^\top P_0(t)C(\cd,t)\d_t+\d_t^\top Q(t)\d_t=0,\\
&P_0(T)= G.
\end{aligned}\right.\ee
Moreover, $P_0(\cd)$ satifies
\bel{}
 P_0(t)(\bx(\cd)|_{[t,T]},\bx(\cd)|_{[t,T]})\ges 0,\q\forall\bx\in C([0,T];\dbR^n),
\ee
which, together \ref{ass:H4}, implies
\bel{}
R(t)+D(\cd,t)^\top P_0(t)D(\cd,t)\ges \a I_m,\q t\in[0,T].
\ee
For $i=0,1,2,...$, define
\begin{align}
\Psi_i(t)=
[D(\cd\,,t)^\top P_i(t)D(\cd\,,t)+R(t)]^{-1}[D(\cd\,,t)^\top P_i(t)C(\cd\,,t)\d_t+B(\cd\,,t)^\top P_i(t)],\label{def-Psi-i}
\end{align}
with $P_i(\cd)$ being the unique solution to the following Lyapunov  equation:
\bel{theorem:well-posedness-RE-proof2}\left\{\begin{aligned}
&\dot{P}_{i+1}(t)+P_{i+1}(t)A(\cd\,,t)\d_t+\d_t^\top A(\cd\,,t)^\top P_{i+1}(t)-P_{i+1}(t)B(\cd,t)\Psi_{i}(t)\\
&\q-\Psi_{i}(t)^\top B(\cd,t)^\top P_{i+1}(t)+\d_t^\top C(\cd\,,t)^\top P_{i+1}(t)C(\cd\,,t)\d_t-\d_t^\top C(\cd,t)^\top P_{i+1}(t) D(\cd,t)\Psi_{i}(t)\\
&\q-\Psi_{i}(t)^\top D(\cd,t)^\top P_{i+1}(t)C(\cd,t)\d_t+\Psi_{i}(t)^\top D(\cd\,,t)^\top P_{i+1}(t)D(\cd\,,t)\Psi_{i}(t)\\
&\q+\d_t^\top Q(t)\d_t+\Psi_{i}(t)^\top R(t)\Psi_{i}(t)=0,\\
&P_{i+1}(T)= G.
\end{aligned}\right.\ee
Then by \autoref{lem:Solvability-RE} again,  we have
\bel{}
P_i(t)(\bx(\cd)|_{[t,T]},\bx(\cd)|_{[t,T]})\ges 0,\q\forall\bx\in C([0,T];\dbR^n),\q i\ges 0.
\ee
For $i\ges 1$, denote
\bel{}
\D_i(\cd)=P_i(\cd)-P_{i+1}(\cd),\q \Pi_i(\cd)=\Psi_i(\cd)-\Psi_{i-1}(\cd).
\ee
Then by \rf{theorem:well-posedness-RE-proof2}, we get
\bel{theorem:well-posedness-RE-proof3}\left\{\begin{aligned}
&\dot{\D}_{i}(t)+\D_{i}(t)A(\cd\,,t)\d_t+\d_t^\top A(\cd\,,t)^\top\D_{i}(t)-\D_{i}(t)B(\cd,t)\Psi_{i}(t)-\Psi_{i}(t)^\top B(\cd,t)^\top \D_{i}(t)\\
&\q+P_{i}(t)B(\cd,t)\Pi_{i}(t)+\Pi_{i}(t)^\top B(\cd,t)^\top P_{i}(t)+\d_t^\top C(\cd\,,t)^\top \D_{i}(t)C(\cd\,,t)\d_t\\
&\q-\d_t^\top C(\cd,t)^\top\D_{i}(t) D(\cd,t)\Psi_{i}(t)-\Psi_{i}(t)^\top D(\cd,t)^\top \D_{i}(t)C(\cd,t)\d_t\\
&\q+\d_t^\top C(\cd,t)^\top P_{i}(t) D(\cd,t)\Pi_{i}(t)+\Pi_{i}(t)^\top D(\cd,t)^\top P_{i}(t)C(\cd,t)\d_t\\
&\q+\Psi_{i}(t)^\top D(\cd\,,t)^\top \D_{i}(t)D(\cd\,,t)\Psi_{i}(t)+\Psi_{i-1}(t)^\top D(\cd\,,t)^\top P_{i}(t)D(\cd\,,t)\Psi_{i-1}(t)\\
&\q -\Psi_{i}(t)^\top D(\cd\,,t)^\top P_{i}(t)D(\cd\,,t)\Psi_{i}(t)+\Psi_{i-1}(t)^\top R(t)\Psi_{i-1}(t)-\Psi_{i}(t)^\top R(t)\Psi_{i}(t)=0,\\
&\D_{i}(T)= 0.
\end{aligned}\right.\ee
Note that
\begin{align}
&P_{i}(t)B(\cd,t)\Pi_{i}(t)+\Pi_{i}(t)^\top B(\cd,t)^\top P_{i}(t)+\d_t^\top C(\cd,t)^\top P_{i}(t) D(\cd,t)\Pi_{i}(t)\nn\\
&+\Pi_{i}(t)^\top D(\cd,t)^\top P_{i}(t)C(\cd,t)\d_t+\Psi_{i-1}(t)^\top D(\cd\,,t)^\top P_{i}(t)D(\cd\,,t)\Psi_{i-1}(t)\\
&-\Psi_{i}(t)^\top D(\cd\,,t)^\top P_{i}(t)D(\cd\,,t)\Psi_{i}(t)+\Psi_{i-1}(t)^\top R(t)\Psi_{i-1}(t)-\Psi_{i}(t)^\top R(t)\Psi_{i}(t)\nn\\
&\q=P_{i}(t)B(\cd,t)\Pi_{i}(t)+\Pi_{i}(t)^\top B(\cd,t)^\top P_{i}(t)+\d_t^\top C(\cd,t)^\top P_{i}(t) D(\cd,t)\Pi_{i}(t)\nn\\
&\qq+\Pi_{i}(t)^\top D(\cd,t)^\top P_{i}(t)C(\cd,t)\d_t+\Pi_i(t)^\top[D(\cd\,,t)^\top P_i(t)D(\cd\,,t)+R(t)]\Pi_i(t)\nn\\
&\qq-\Pi_i(t)^\top[D(\cd\,,t)^\top P_i(t)D(\cd\,,t)+R(t)]\Psi_i(t)-\Psi_i(t)^\top[D(\cd\,,t)^\top P_i(t)D(\cd\,,t)+R(t)]\Pi_i(t)\nn\\
&\q=\Pi_i(t)^\top[D(\cd\,,t)^\top P_i(t)D(\cd\,,t)+R(t)]\Pi_i(t)\ges 0.
\end{align}
Then by \autoref{lem:Solvability-RE} again, \rf{theorem:well-posedness-RE-proof3} admits a unique solution $\D_i(\cd)$ satisfying $\D_i(\cd)\ges 0$.
It follows that
\bel{theorem:well-posedness-RE-proof4}
 P_i(t)(\bx(\cd)|_{[t,T]},\bx(\cd)|_{[t,T]})\ges P_{i+1}(t)(\bx(\cd)|_{[t,T]},\bx(\cd)|_{[t,T]})\ges 0,\q\forall\bx\in C([0,T];\dbR^n),\q  i\ges 0.
\ee
Thus, for any $\bx\in C([0,T];\dbR^n)$ and $t\in[0,T]$, $ P_i(t)(\bx(\cd)|_{[t,T]},\bx(\cd)|_{[t,T]})$ is a decreasing sequence,
which is convergent as $i\to\i$.
Denote
\bel{theorem:well-posedness-RE-proof5}
 P(t)(\bx(\cd)|_{[t,T]},\bx(\cd)|_{[t,T]})\deq\lim_{i\to\i} P_i(t)(\bx(\cd)|_{[t,T]},\bx(\cd)|_{[t,T]}),\q \forall \bx\in C([0,T];\dbR^n),\q t\in[0,T].
\ee
By \autoref{lem:norm}, $P(\cd)$ can be uniquely extended to $C([0,T];\cS^n)$ by
\begin{align}
 P(t)(\bx_t(\cd),\bx_t^\prime(\cd))\deq{ P(t)\big (\bx_t(\cd)+\bx_t^\prime(\cd),\bx_t(\cd)+\bx_t^\prime(\cd)\big)
-P(t)\big (\bx_t(\cd)-\bx_t^\prime(\cd),\bx_t(\cd)-\bx_t^\prime(\cd)\big)\over 4},\nn\\
\q \forall \bx_t,\bx_t^\prime\in C([t,T];\dbR^n).
\label{theorem:well-posedness-RE-proof5-1}
\end{align}
Then using the fact that $P_i(\cd)$ is symmetric, from \rf{theorem:well-posedness-RE-proof5}--\rf{theorem:well-posedness-RE-proof5-1} we have
\bel{}
 P(t)(\bx(\cd)|_{[t,T]},\bx^\prime(\cd)|_{[t,T]})=\lim_{i\to\i} P_i(t)(\bx(\cd)|_{[t,T]},\bx^\prime(\cd)|_{[t,T]}),\q \forall \bx,\bx^\prime\in C([0,T];\dbR^n),\q t\in[0,T],
\ee
which, together with \rf{def-Psi-i}, implies that
\bel{}
\lim_{i\to\i} \Psi_i(t)=[D(\cd\,,t)^\top P(t)D(\cd\,,t)+R(t)]^{-1}[D(\cd\,,t)^\top P(t)C(\cd\,,t)\d_t+B(\cd\,,t)^\top P(t)],\q\forall t\in [0,T].
\ee
Moreover, by \autoref{lem:norm} and \rf{theorem:well-posedness-RE-proof4}, we have
\bel{}
\sup_{t\in [0,T]}\|P_i(t)\|\les 2\sup_{t\in [0,T]}\|P_0(t)\|<\i,\q i\ges 0,
\ee
and then
\bel{}
\sup_{t\in[0,T],i\ges 0}\|\Psi_i(t)\|<\i.
\ee
Taking $i\to\i$ in \rf{theorem:well-posedness-RE-proof2},
we see that $P(\cd)$ satisfies \rf{theorem:well-posedness-RE-proof1},
and hence \rf{Riccati-Equation-RW}.

\ms

By \autoref{theorem:well-posedness-RE}, we show that under \ref{ass:H1}--\ref{ass:H2} and \ref{ass:H4},
 the decoupling field of the optimality
system associated with Problem (LQ-FSVIE) really exists.
Combining \autoref{thm:HJB-Riccati} and \autoref{theorem:well-posedness-RE},
we get the following result.
It is noteworthy that the solvability of path-dependent HJB equation \rf{HJB-PPDE2} is proved by an analytic method.
%We remark that the uniqueness of the classical solution to  \rf{HJB-PPDE} is due to the uniqueness
%of the value function of Problem (LQ-SV).

\begin{corollary}
Let {\rm\ref{ass:H1}--\ref{ass:H2}} and {\rm\ref{ass:H4}} hold.
Then the path-dependent HJB equation \rf{HJB-PPDE2} admits a  classical solution.
\end{corollary}

%\section{The case with nonsmooth coefficients}

\section{Conclusion}
The main contribution of this paper is that we develop a decoupling method for
the  optimality system  associated with Problem (LQ-FSVIE)
and then represent the optimal control as a causal feedback of  the state process.
The key technique is establishing a link between the type-I and type-II BSVIEs,  deriving
the associated path-dependent Riccati equation, and proving the solvability of this new type of Riccati equations.
We believe that this paper can open the door to the Riccati-equation approach for the LQ problem of SVIEs.
Applying the mollification method, the main results obtained in the paper
still hold true for the case with some nonsmooth (or singular) coefficients;
Using the method developed in the paper, the framework can be further extended to the indefinite LQ optimal control problems for SVIEs,
the LQ  game problems for SVIEs, the LQ control/game problems for mean-field SVIEs, etc.
A still challenging problem is  the LQ optimal control problem for SVIEs with random coefficients.
Moreover, we hope that our approach of decoupling the optimality systems  and establishing the
well-posedness of the path-dependent Riccati (or HJB) equation can help us to solve some nonlinear problems.
We will report the related results in our future publications.

\end{document}